\documentclass[a4paper,11pt]{article}

\usepackage{amssymb}
\usepackage{mathrsfs}
\usepackage{amsmath}
\usepackage{amsthm}
\usepackage{enumerate}
\usepackage{enumerate}
\usepackage{color}
\usepackage{verbatim}
\usepackage{todonotes}

\usepackage[margin=3cm, includefoot, footskip=30pt]{geometry}

\usepackage{array}
\newcolumntype{e}{>{\displaystyle}r @{\,} >{\displaystyle}c @{\,} >{\displaystyle}l}

\usepackage{amsfonts,amssymb,amsopn,amscd}

\usepackage[colorlinks]{hyperref}
\hypersetup{final}
\usepackage{xcolor}
\hypersetup{
    colorlinks,
    linkcolor={blue!40!black},
    citecolor={blue!40!black},
    urlcolor={blue!80!black}
}
\usepackage{nameref}

\usepackage[figure,table]{hypcap} 
\hypersetup{
   bookmarksnumbered,
   pdfstartview={FitH},
   citecolor={blue},
   linkcolor={blue},
   urlcolor={black},
   pdfpagemode={UseOutlines}
}

\theoremstyle{plain}
\newtheorem{Theorem}{Theorem}[section]
\newtheorem{Corollary}[Theorem]{Corollary}
\newtheorem{Lemma}[Theorem]{Lemma}
\newtheorem{Proposition}[Theorem]{Proposition}

\theoremstyle{definition}

\newtheorem{Definition}[Theorem]{Definition}
\newtheorem{Remark}[Theorem]{Remark}

\newcounter{Condition}

\renewcommand{\P}{\mathbb{P}}
\newcommand{\E}{\mathbb{E}}
\newcommand{\R}{\mathbb{R}}
\newcommand{\N}{\mathbb{N}}
\newcommand{\Z}{\mathbb{Z}}

\newcommand{\e}{\varepsilon}

\newcommand{\1}{{\text{\Large $\mathfrak 1$}}}

\def\bs{\backslash}

\def\bP{\mathbb{P}}
\def\bE{\mathbb{E}}
\def\cE{\mathcal{E}}

\def\cS{\mathscr{S}}
\def\sS{\mathscr{S}}

\def\cT{\mathcal{T}}

\def\reff#1{(\ref{#1})}

\numberwithin{equation}{section}

  \newcounter{constant}

\usepackage{calc}
\def\arraypar#1{\parbox[c]{\textwidth - 2cm}{\centering #1}}

\usepackage{environ}
\NewEnviron{display}{
\begin{equation}\begin{array}{c}
  \arraypar{\BODY}
\end{array}\end{equation}
}

\begin{document}

\title{Internal Diffusion Limited Aggregation \\ with Critical Branching Random Walks}

\author{
Amine Asselah\footnote{amine.asselah@u-pec.fr;  Universit{\'e} Paris-Est, LAMA, UMR805D CNRS, Paris, France} \and
Vittoria Silvestri\footnote{silvestri@mat.uniroma1.it; University of Rome La Sapienza, Dept.\ of Mathematics, Rome, Italy} \and 
Lorenzo Taggi\footnote{lorenzo.taggi@uniroma1.it; University of Rome La Sapienza, Dept.\ of Mathematics, Rome, Italy}
}

\maketitle

\begin{abstract} 
Internal Diffusion Limited Aggregation is an interacting particle system that describes the growth of a random cluster governed by the boundary harmonic measure seen from an internal point.  Our paper  studies  IDLA in $\mathbb{Z}^d$ driven by critical branching random walks.
We prove that,  unlike classical IDLA, this process exhibits a phase transition in the dimension.  More precisely,  we establish the existence of a spherical shape theorem in dimension $d \geq 3$ and the absence of a spherical shape theorem for $d \leq 2$. Our bounds on the inner and outer worst deviations are of polynomial nature, which we expect to be a feature of this model. 
\end{abstract}

\setcounter{tocdepth}{2}
\tableofcontents

\section{Introduction}\label{sec:introduction}
Internal Diffusion Limited Aggregation (IDLA) is a mathematical model aimed at describing erosion processes driven by diffusing particles. In its classical version on the cubic lattice $\mathbb Z^d$, the process starts from an aggregate made of a single site, the origin $\{ 0 \}$. We say that the origin is \emph{occupied}, and the other sites are \emph{empty}. At each iterative step, a new diffusing particle is started from the origin. The particle moves as a simple random walk on $\mathbb{Z}^d$ until it finds an empty site, where it stops forever. We say that the particle \emph{settles} at that site, and the site is added to the current aggregate. 
Equivalently, the aggregate grows by swallowing boundary sites, one at the time, according to the \emph{harmonic measure on the boundary seen from the origin}, namely the exit law from the aggregate of a simple random walk starting at $\{0\}$.

This model was introduced by Meakin and Deutch as an erosion process  in 1986  \cite{meakin1986formation}, and independently by Diaconis and Fulton as an Abelian process  in 1991 \cite{diaconis1991growth}. We will come back later to the meaning of Abelianity in this context. 
Right from its introduction, the first question to be addressed on IDLA was that of a shape theorem, namely whether rescaled IDLA aggregates approach a deterministic shape. This  was established by Lawler, Bramson and Griffeath  in 1992 \cite{lawler1992internal}: the shape is asymptotically an Euclidean ball on $\mathbb{Z}^d$ for any $d\geq 1$. Nearly two decades later, IDLA maximal deviations from circularity were shown to be logarithmic in the radius of the limiting ball in any dimension $d \geq 2$ by Asselah and Gaudilli{\`e}re and, independently, Jerison, Levine and Sheffield \cite{asselah2013logarithmic,asselah2013sublogarithmic,jerison2012logarithmic,jerison2014internal,jerison2013internal}.

In this paper we introduce a new variant of IDLA, in which the dynamics is driven by \emph{Critical Branching Random Walks} in place of simple random walks. 
Recall that, given a probability distribution $\nu $ on $\mathbb N = \{0,1,2 \dots \}$ with strictly positive variance,  a  Branching Random Walk (BRW) on $\mathbb{Z}^d$ starting from a vertex $z \in \mathbb{Z}^d$ is defined as follows. At time zero a single particle is placed at $z$. The particle has an exponential clock: when the clock rings, the particle splits into 
 a random number of indistinguishable particles at the same site (its \emph{children}) with law $\nu$, which each  take an independent and instantaneous simple random walk step on $\mathbb Z^d$. After that, each particle continues evolving as its parent, independently of the others. We refer the reader to Section \ref{sec:BRW} for a formal definition.  If $\nu$ has mean $1$, the BRW is said to be \emph{critical}, and it is known that the associated branching process dies out in finite time almost surely. 

Here we propose to use critical BRWs to drive IDLA dynamics. The (discrete time)  BIDLA process $(A(t))_{t\in \mathbb{N}} $ is defined as follows. Set $A(0) =\emptyset $ and $A(1) = \{0\}$. For all $t\geq 1$, the aggregate (or \emph{cluster}) $A(t)$ is a finite connected subset of $\mathbb Z^d$. We say that sites in $A(t)$ are \emph{occupied}, while sites outside $A(t)$ are \emph{empty}.  For each $t\geq 1$ define $A(t+1)$ from $A(t)$ inductively, by 
releasing a critical BRW from the origin, with the rule that when a particle reaches an empty site the particle stops there forever (in which case we say it \emph{settled}). The site is then declared \emph{occupied}, and the following particles stepping on it will continue their evolution until reaching a new empty site. Since the BRW is critical, after finitely many steps all particles will have settled. We call the resulting configuration the \emph{BIDLA stabilization} of $A(t) + \1_0$, and denote it by $\mathscr{S}(A(t) + \1_0 )$. Then define 
	\[ A(t+1) := \mathscr{S} ( A(t) + \1_0 ) . \]
In words, at each discrete time step $t\geq 1$ a critical BRW is released from the origin in $A(t)$ with settling upon visiting empty sites, and the settling locations are added to the cluster to make $A(t+1)$. 
See Figure \ref{fig:BIDLA} below for simulations of BIDLA clusters on $\mathbb Z^2$, and Section \ref{sec:BIDLA} for a formal definition of the model. 
\begin{figure}[!!ht]
    \begin{center}
     \centering
  \mbox{\hbox{
  \includegraphics[width=.55\textwidth]{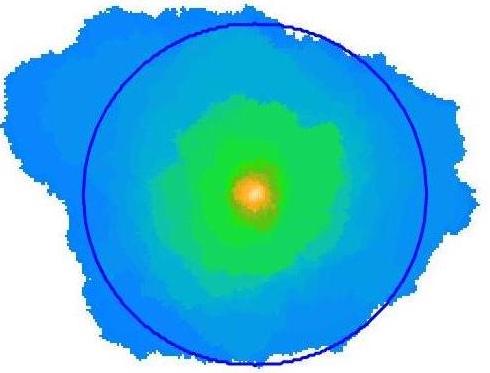}
   \;  \includegraphics[width=.43\textwidth]{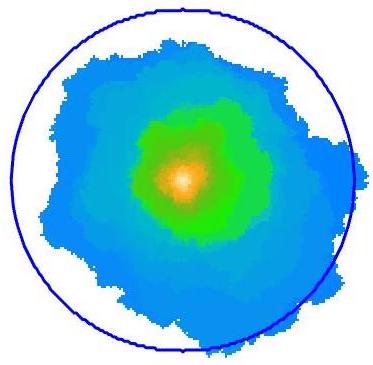}}}
  \\ \vspace{2mm}
  \mbox{\hbox{
  \quad \quad \includegraphics[width=.44\textwidth]{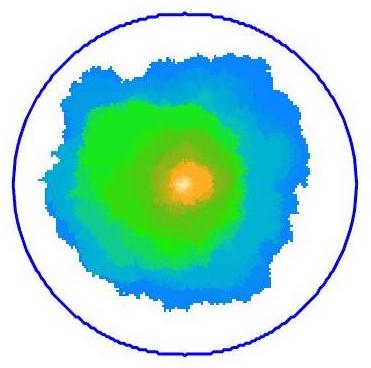}
    \quad \quad \includegraphics[width=.44\textwidth]{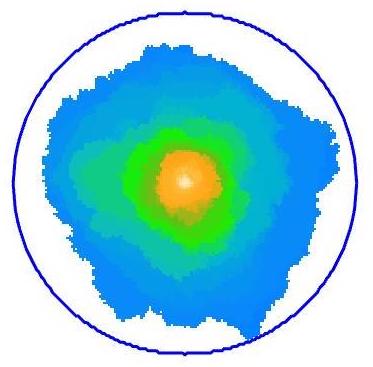}}}
       \end{center}
 \caption{Four simulations of BIDLA clusters on $\mathbb{Z}^2$ at time $t=20000$, plotted together with a disc of radius $\sqrt{t/\pi}$. Here the different colours represent the arrival time of particles (where time is indexed by particle releases).} \label{fig:BIDLA}
 \end{figure}

The phenomenology of Branching Random Walks with settling, and hence of BIDLA,  is drastically different from that of  simple random walks. 
To start with, notice that while in classical IDLA launching one particle from the origin results in the addition of exactly one particle to the aggregate, in BIDLA new particles are created and some are lost due to the nature of the driving Branching Random Walks. Thus, even though 
	\[ \bE ( |A(t) \setminus A(t-1) | ) =1, \]
(see Lemma  \ref{lem-growth} below), in a single time step the cardinality of a cluster may, a priori, grow by an arbitrarily large  amount.  Furthermore, the exit locations of the arbitrarily many particles which are added to the cluster in one time step are correlated, and may be clumped together, with particles stepping on each other to reach sites which are far away from the cluster. These two new and crucial features of BIDLA play a fundamental role in the spacial growth of clusters, and cause classical approaches to the proof of shape theorems to fail in this setting. 
To tackle this issue we introduce a new auxiliary process of independent interest  called \emph{Random Barrier Growth}, which can be coupled with BIDLA and whose growth is easier to bound. 


\subsection{Statement of results}
We address the question of the existence of a deterministic  limiting shape for BIDLA aggregates. 
Our results hold under the following assumption on the offspring distribution $\nu$ on $\mathbb{N}$, which is in force throughout the paper: 
	\begin{equation}\label{assumption}
	\tag{H} 
	\sum_{k=0}^\infty k \nu (k) = 1 , 
	\qquad 
	\sigma^2 := \sum_{k=0}^\infty (k -1 )^2 \nu (k) >0 , 
	\qquad 
	\exists \lambda >0 \; : \; \,\sum_{k=0}^\infty e^{\lambda k } \nu (k) < \infty . 
	\end{equation}
For $r>0$  let $B_r := \{ z \in \mathbb Z^d : \| z\| < r \}$
denote the Euclidean ball of radius $r$ around the origin. 
Then for any $\varepsilon >0$ denote by 
$ \mathcal{S}_\varepsilon $  the set of $\e$-symmetric clusters, defined as 
\begin{equation}\label{def:symmetric}
	 \mathcal{S}_\varepsilon := \big\{ A \subseteq \mathbb Z^d : \ \exists r>0 \mbox{ s.t. } 
	B_{(1-\e ) r} \subseteq A \subseteq B_{(1+\e ) r} \big\} . 
\end{equation}
Our main result is the following. 
\begin{Theorem} \label{th:main}
Let $(A(t))_{t\in \N }$ denote a BIDLA process on $\mathbb Z^d$ satisfying Assumption \eqref{assumption}. 
\begin{itemize}
\item If $d=1,2$ then there exists $\e >0$ such that 
	\[ \bP \big( A(t) \not\in \mathcal{S}_\e \mbox{ infinitely often} \big) =1 . \]
\item If $d\geq 3$ then $\forall \e >0$ 
	\[ \bP \big( A(t) \in \mathcal{S}_\e \mbox{ eventually} \big) =1. \]
\end{itemize}
\end{Theorem}
This shows that BIDLA exhibits a {\it phase transition}, in that the existence of a spherical  shape theorem depends on the dimension. 
If $d\geq 3$ the deterministic limiting shape is the same as classical IDLA on $\mathbb Z^d$. If $d=1$ the barycenter of the aggregate (which on $\mathbb Z $ is simply an interval) moves macroscopically infinitely often. Finally, if $d=2$ our results would prove that there is no shape theorem provided one could rule out the existence of asymptotic shapes other than Euclidean balls.

In dimension three and more we are able to prove stronger estimates for the maximal deviations from circularity. To state them, for any $t\geq 1$ let $r(t)$ denote the radius of a Euclidean ball of volume $t$ centered at the origin, so that by the shape theorem $A(t) $ is close to $B_{r(t)}$ eventually in $t$, almost surely. 
We define 
	\[ \begin{split} 
	\delta^I (t) &  := \inf \{ \delta \geq 0 : B_{r(t)- \delta } \subseteq A(t) \}
	\\
	\delta^O (t) &  := \inf \{ \delta \geq 0 : A(t) \subseteq B_{r(t) +\delta } \} . 
	\end{split}
	\]
Then  $\delta^I(t)$ and $\delta^O(t)$ measure respectively the inner and outer maximal deviation from circularity of $A(t)$.

\begin{Theorem} \label{th:deviation}
Let $(A(t))_{t\geq 0}$ denote a BIDLA process in  dimension $d>2$ satisfying Assumption \eqref{assumption}. Then for any $\epsilon>0$,
\begin{equation}\label{main-error}
\bP \Big( \Big\{ \delta^I(t)> t^{\frac{1}{2}+\epsilon} \Big\} 
\cup \Big\{\delta^O(t)> t^{1-\frac{1}{2d}+\epsilon} \Big\} 
\, \, \text{infinitely often in }t \Big)
=0.
\end{equation}
\end{Theorem}
We conclude with  a deviation result which is interesting on its own, expressing the fact that a number of particles proportional to the volume of a ball  $B_n$ covers it with overwhelming probability when $d > 2$,
regardless of their specific placement inside $B_{n/2}$.
Moreover, the same holds with uniformly positive probability when $d=2$.

Let $\eta : \mathbb{Z} \mapsto \mathbb{N} $ be an arbitrary particle configuration supported inside the ball $B_{n/2}$ for some $n \geq 1$, where $\eta (z)$ denotes the number of particles at $z$ for all $z \in \Z^d$. Recall that $\mathscr{S}(\eta )$ denotes the BIDLA stabilization of $\eta$.
For a subset $K $ of $\mathbb{Z}^d$ we denote by $|K| $ the cardinality of $K$, while if $\eta $ is a particle configuration we let
	\[ |\eta| := \sum_{x \in \mathbb{Z}^d} \eta(x) \]
denote the total number of particles in $\eta$. 
\begin{Lemma}\label{lem-covering}
For any $d>2$  there exists a constant  $\alpha $ large enough such that for any integer $n$ and any particle configuration $\eta$ supported in $B_{n/2}$ with $|\eta| \geq  \alpha  |B_n|$, it holds that 
\begin{equation}\label{ineq-covering1}
\bP\big(B_n\not\subseteq \mathscr{S}(\eta ) \big)\le
 \exp\Big(-C \frac{\alpha\cdot n}{\log n\cdot \1_{d=3}+\1_{d>3}}\Big)
\end{equation}
for some positive and finite constants  $C$,  depending only on the dimension.

Moreover, for any $d \geq 2$ there exists a constant  $\alpha $ large enough such that for any integer $n$ and any particle configuration $\eta$ supported in $B_{n/2}$ with $|\eta| \leq  \alpha \cdot |B_n|$, it holds that 
\begin{equation}\label{ineq-covering2}
\bP\big(B_n\not\subseteq \mathscr{S}(\eta ) \big)\geq 
\exp\Big( - c \cdot \alpha\cdot n^{d-2}\Big)
\end{equation}
for some positive and finite constants  $c$,  depending only on the dimension. 
\end{Lemma}

In contrast, for classical  IDLA, under the same hypotheses, the upper bound reads 
\begin{equation}\label{srw-covering}
\bP\big(B_n\not\subseteq \mathscr{S} (\eta )  \big)\le
\exp\Big(-C \frac{\alpha\cdot n^2}{\log n\cdot \1_{d=2}+\1_{d>2}}\Big) 
\end{equation}
(see Lemma 1.3 of \cite{asselah2013sublogarithmic}). 
This shows that, even though in dimension $d=3$ both classical IDLA and BIDLA have a spherical shape theorem, the large deviation behaviours of the two processes differ significantly.

\subsection{Sketch of the proof}
We include here a brief overview of the proof of Theorem \ref{th:main}. 

\subsubsection{No shape theorem in low dimension}
In order to show that BIDLA is infinitely often  asymmetric  almost surely, it suffices to prove that once in a symmetric state, the probability that the process gets to an asymmetric state in one step\footnote{Here we consider the jump process associated to BIDLA, so that the cluster changes at every step, as explained at the beginning of Section \ref{sec:d2}.} is uniformly positive. To this end, we show that when launching one BRW from the origin conditioned on exiting the current cluster, the following happens with positive probability: 
\begin{itemize}
\item[(i)] If the cluster is close to a ball $B_n$, the BRW reaches the inner boundary of $B_n$ with order $n^2$ particles. 
\item[(ii)] These particles are clumped together.
\item[(iii)] Settling the localised particles creates the claimed asymmetry. 
\end{itemize}
Points (i) and (ii) are properties of a single BRW conditioned on travelling to distance $n$ from its starting point, and are proved in Section \ref{sec:proofd2}.  The third statement (iii) accounts for settling, and it is proved by combining the classical approach of Lawler, Bramson and Griffeath with new second moment estimates on the local times of BRW. The necessity of controlling higher moments of the local times, rather than only the first moment, is a feature of Branching IDLA, and it stems from the randomness of the number of individuals during the time evolution of a single BRW. The needed second moment estimates are proved in Section~\ref{sec:Green}.

\subsubsection{Shape theorem in high dimension}
Our proof of the shape theorem for $d>2$ consists of three steps, namely: 
\begin{itemize}
\item[Step 1.] Establishing an \emph{inner bound}, that is proving that the aggregate contains a Euclidean ball of roughly the right radius, even when particles are \emph{frozen} upon reaching the boundary of this ball. 
\item[Step 2.] Bounding the number of frozen particles on the boundary of the filled ball. 
\item[Step 3.] Releasing the frozen particles in \emph{waves}, and iteratively showing that each wave results in the settling of a positive proportion of the particles. This sets up a \emph{contraction} in the number of particles which, after sufficiently many iterations,  gives the desired outer bound. 
\end{itemize}
Step 1 is carried out in Section \ref{sec:innerbound} following the classical approach by Lawler, Bramson and Griffeath \cite{lawler1992internal} together with an independence property observed first by Asselah and Gaudilli\`ere in \cite{asselah2013logarithmic}. 
We then take care of Step 2 in Proposition \ref{prop-standby}, where we show that, on the event that the inner ball is filled (which has high probability by Step 1), only a small proportion of particles is frozen on the boundary of the ball. This requires the finite exponential moment assumption \eqref{assumption} on the offspring distribution $\nu$, which is in force throughout the paper. 

The main technical novelty in our approach comes in Step 3, where we have to show that releasing the frozen particles does not cause the cluster to grow macroscopically. This is typically easy for classical IDLA (see the robust approaches of \cite{lawler1992internal} and \cite{duminil2013containing}), but for Branching IDLA it becomes far from trivial due to the creation/loss of particles, as well as the fact that close-by particles can help each other travel far within one BIDLA step. To control the growth of the aggregate we introduce a new auxiliary process called \emph{Random Barrier Growth},   which is both easier to control than BIDLA and it \textit{grows  faster} than BIDLA itself. Hence, thanks to the Abelian property, from an outer bound on the auxiliary process, we can deduce an outer bound for BIDLA clusters. 

In \emph{Random Barrier Growth} settling is only allowed on spherical layers.
This restriction, while artificial at first sight, turns out to be a powerful tool,  because the Green's function of a random walk killed when leaving a ball is much better understood,   thereby giving us a precise handle on the number of particles that settle on the layer. 
A further crucial ingredient is that these spherical layers are not fixed in advance, but are sampled at random according to a carefully chosen distribution.  
This randomness plays a decisive role in providing an upper bound for the number of particles which do not settle on the layer (see also Remark \ref{rem:layer}).
Iterating this construction for successive layers, one obtains a contraction mechanism which ultimately yields the outer bound for BIDLA clusters. 
This construction is carried out in Section \ref{sec:outerbound}, with the crucial contraction bound stated in Proposition \ref{prop:settlmentsphericalprocess}.

\subsection{Related models and open questions}
As mentioned above, Internal DLA aggregates, or \emph{clusters}, grow according to the harmonic measure on the cluster's boundary seen from an \emph{internal} point (the origin). The dynamics is therefore smoothing, in that arms grow slowly and fjords are filled quickly. This makes the shape theorem stable under small changes of the model, such as modifying the base graph or the transition probabilities of the driving simple random walks. See e.g.\  \cite{duminil2013containing,lucas2014limiting,silvestri2020internal,bou2024internal} and references therein for some results in this direction. We also mention a recent work by 
Benjamini, Duminil-Copin, Kozma and Lucas, who 
introduced a variant of IDLA where the starting point of a new random walk is chosen uniformly at random inside the aggregate, thereby proving that the limiting shape  remains  spherical  \cite{benjamini2017internal}. 

If, on the other hand, rather than using the internal harmonic measure  one grows clusters according to the \emph{external harmonic measure} on the boundary, i.e.\ the hitting measure of a simple random walk starting far away from the aggregate, one obtains \emph{Diffusion Limited Aggregation}, in short DLA, a model which is notoriously hard to analyse \cite{witten1981diffusion,derrida1992needle,hastings1998laplacian}. DLA clusters are fractal-looking, with a few arms growing faster than the bulk. The mechanism for this reinforcement is still escaping a rigorous understanding, the only available result on $\mathbb Z^d$ dating back to 1987, when Kesten proved an upper bound for the growth rate of DLA clusters \cite{kesten1987long}.

In order to tame the reinforcement leading to fractality, one could grow clusters according to a \emph{fractional} power of the external harmonic measure. This defines \emph{Dielectric Breakdown Models}, and it reduces to the \emph{Eden Model} when this power vanishes \cite{niemeyer1984fractal,eden1961two}. While the Eden model is well studied due to its coupling with \emph{First Passage Percolation}, a mathematical  description of Dielectric Breakdown aggregates is  still far out of reach. We mention the recent work of Losev and Smirnov \cite{losev2025long}, which provides  bounds for the clusters' growth rate \`a la Kesten.  
Recent progress on the analysis of two-dimensional random aggregation models where the growth is driven by a fractional power of the (external) harmonic measure was achieved by working in the continuum, where clusters can be described via random conformal maps, and the correlation strength between those maps determines the boundary growth rate. See \cite{hastings1998laplacian,norris2023scaling,norris2024stability} and references therein on Hastings-Levitov models and Aggregate Loewner Evolution. 

From the point of view of particle systems, Internal DLA belongs to a larger family of models called 
\emph{Activated Random Walks}, where settling is neither eternal nor instantaneous: 
a particle stops (or settles) at an exponential time of rate $\lambda$, called the \emph{sleep parameter},  and it regains activity if another particle steps on its settling site. The dynamics of an ARW system depends on the parameter $\lambda$ and on the initial configuration of particles. Classical IDLA corresponds to $\lambda=\infty$ (particles settle, or fall asleep, instantaneously when at an empty site), and all particles started at the origin. When $\lambda < \infty $ the analysis is more difficult due to the long-range correlations and the non-monotonicity of the system. According to the values of the sleep parameters and particle density, ARW in infinite volume  undergoes a \emph{phase transition} from eternal activity to stabilization, which has been object of extensive research \cite{rolla2012absorbing,rolla2019universality}. The main interest in this model lies in the fact that it is widely believed to exhibit \emph{Self Organised Criticality} and \emph{Universality} \cite{bak1988self,rolla2020activated,levine2024universality}, which was recently proved to be the case in dimension $d=1$ by Hoffman, Johnson and Junge  \cite{hoffman2024density}. 
One could also look at ARW dynamics as a growth process by releasing finitely many particles from the origin and asking whether the visited region until stabilization approaches a deterministic shape. A first result in this direction was obtained in dimension one by Levine and Silvestri \cite{levine2021far}, followed by the proof of a shape theorem by Hoffman, Johnson and Junge \cite{hoffman2024density}. 

The problem of describing the aggregate’s growth has also been studied in the so-called Oil and Water model~\cite{bond_levine_abelian_networks_I}, which is conjectured to satisfy a spherical shape theorem as in classical IDLA, though with a slower growth rate.
This conjecture has been partially resolved in one dimension by Candellero, Ganguly,  Hoffman,  Levine~\cite{hoffman_candellero_ganguly_levine_oil_and_water}.
As shown in~\cite{candellero_stauffer_taggi_oil_and_water}, 
the model can be described in terms of  BRW whose offspring law is random,  mainly critical but   occasionally  subcritical.
We believe  that the techniques introduced in the present article could  contribute to further progress on  Oil and Water  and related models.

\subsubsection{Open questions}
We conclude with some open questions which we believe deserve further investigation. 
In dimensions $d>2$, where a shape theorem is available, a central challenge is to determine the size of the fluctuations. 
Unlike classical IDLA, we expect these fluctuations to grow on a polynomial scale, and pinning down their precise order would provide a deeper understanding of the model.   
In two dimensions it would be very interesting to exclude the existence of a deterministic non-spherical limiting shape, thereby confirming a fundamental qualitative difference with the classical case. Is there a Markovian dynamics defined directly in the continuum which could be seen as the evolution in time of rescaled BIDLA in $d=2$? 
Finally, the behaviour under sub-critical offspring distributions~$\nu$ remains largely unexplored. In this case we expect the limiting shape to lose the spherical symmetry, as sub-critical BRWs conditioned to survive up to a large distance from their origin would be ballistic rather than diffusive. 

\subsection*{Organization of the paper} 
The paper is organized as follows. 
In Section~\ref{sec:Models} we provide a formal definition of Branching Random Walks
and Branching IDLA, and fix the notation. 
In Section~\ref{sec:prelim} we present new estimates for the second moment of the local time of BRW restricted to finite sets, and discuss a classical approach to inner bounds for IDLA. 
In Section~\ref{sec:d2} we prove the absence of a spherical shape theorem in  dimension $d=2$. The case $d=1$ follows along the same lines, and we leave it to the reader since easier.  
In Section~\ref{sec:innerbound} we establish an inner bound on the growth of the aggregate for $d>2$, 
and derive an upper bound on the number of particles escaping such a bound. 
In Section~\ref{sec:outerbound} we introduce the new auxiliary process, 
the \emph{Random Barrier Growth}, and 
state our main result providing an outer bound on the growth of this process. 
We then conclude the proof of the shape theorem.

\subsection*{Notation}
We use $\mathbb{N} = \{0, 1, \ldots, \}$ and $\mathbb{Z}_+  = \{1, 2, \ldots\}$.
For any  subset $K \subset \mathbb Z^d$ we denote by
	\[ \partial K := \{ y\in \mathbb{Z}^d \setminus K
	 : y\sim x \mbox{ for some } x \in K \} \]
its exterior boundary. 
We denote by $\| x \|$ the Euclidean norm of $x \in \mathbb{Z}^d$. 
The constants  $c, C$, whose value may vary from line to line, will always be positive and finite, and may depend only on the dimension.   Any other dependence of the constants on other quantities will be mentioned explicitly.  
When a constant needs to be referred to later, we will denote it by a specific number or letter rather than by $c$ or $C$. 
Finally, for any $R>0$ we denote by $B_R$ the Euclidean ball of radius $R$ centred at the origin, i.e.\ $B_R = \{ x \in \Z^d : \| x\|< R\}$, and write $|B_R| =  (\omega_d + o(1)) R^d$ for its cardinality, where $\omega_d$ is the volume of the unit ball in $\R^d$. 

\section{Definition of BIDLA}\label{sec:Models}
We start by recalling the definition of Branching Random Walks, which will later describe the dynamics of the particles, or \emph{explorers}, in BIDLA. 

\subsection{Branching Random Walks} \label{sec:BRW}
We let $\mathbf {\mathcal T}$ be a critical Bienaym\'e-Galton-Watson tree, whose offspring distribution $\nu$ satisfies Assumption \eqref{assumption}.
We think of $\mathbf {\mathcal T}$ as a set of vertices, and call $E(\mathbf {\mathcal T})$ the
edges of $\mathbf {\mathcal T}$, where there is one edge between an individual and each child.
A Branching Random Walk (BRW) is the associated tree-indexed random walk 
 $ (S_u )_{ u\in \mathbf {\mathcal T}}$, where
time is encoded by the tree $\mathbf{\mathcal T}$ and
whose jump distribution is the uniform measure on the neighbours of the origin.
In other words, independent increments $(\xi(e))_{ e\in E(\mathbf {\mathcal T})}$ 
are associated to the edges of the tree, and if
$[\emptyset,u]$ is the sequence of edges between the root $\emptyset$ and vertex $u$, then
we choose an initial position $ z\in \mathbb Z^d$, and set $S_\emptyset=z$ and 
\[
 S_u^z=z+\sum_{e\in [\emptyset,u]}\xi(e)
 \qquad \forall u\in \mathbf {\mathcal T}. 
\]
Given $x\in  \Z^d$, we define the time spent at $x$ by the BRW as
the number of individuals $u$ in the tree such that $S^z_u=x$, that is
\begin{equation}\label{def-localtimes}
\ell^z(x) := \sum_{u\in \mathcal T}\1 ( S^z_u=x) ,   
\end{equation}
and we refer to $\ell^z(x)$ as the \emph{local time} at $x$ of a BRW started at $z$. 
For a finite subset $K \subset \mathbb Z^d$ we denote the local time at $K $ by 
\[ 
\ell^z(K):=\sum_{x\in K}\ell^z(x) . 
\]
For $u,v\in \mathbf{\mathcal T}$, we say that $v$ is an ancestor of $u$, denoted $v\prec u$, 
if $v$ belong to the geodesic between the root and $u$, excluding $u$. If we want to
include $u$, we say $v\preceq u$. 
\subsubsection{BRW restricted to a domain} \label{sec:restricted}
It is convenient to consider the branching random walk
restricted to a domain $K\subseteq \mathbb Z^d$. Let 
\begin{equation}\label{def-stoppedBRW}
\mathcal T^z(K):=\{u\in \mathbf{\mathcal T}:\ \forall v\preceq u,\ S^z_v\in K\}, 
\end{equation}
and observe that $\mathcal T^z(K)\subseteq K$. Denote  the set of individuals on $\partial K$ whose ancestors are all in $K$ by 
\[
\partial_{K} \mathcal T^z:=\{u\in  \mathcal T^z:\ 
 S^z_u\in \partial K \, \mbox{ and }  S^z_v \in K \;\; 
\forall v\prec u\}.
\]
In words, $\cT^z(K)$ is the collection of individuals whose ancestral lines belong to $K$, whereas
$\partial_{K}\cT^z$ is the set of so-called \emph{pioneers}  on $\partial K$, namely those individuals  whose ancestors are all in $K$, but who have escaped $K$, and because of the nearest neighbour nature of
the increments of the BRW, they are on $\partial K$. We can define local times for the restricted
BRW  $\cT^z(K)$, as a set function on the whole of $K\cup\partial K$ as follows: 
\begin{equation}\label{def-localtimes-Lambda}
\begin{split} 
\ell^z_{K}(x) &  := \sum_{u\in \mathcal T^z(K)}\1 ( S^z_u =x ),
\quad \mbox{for }  x\in K, \\
\ell^z_{K}(x)  & := \sum_{u\in \partial_{K}\mathcal T^z}\1 (S^z_u =x),
\quad \mbox{for }  x\in \partial K . 
\end{split}
\end{equation}
The  local time of a collection of independent Branching Random Walks restricted to $K$ starting at locations 
$\eta = (x_1,  \ldots, x_n)$  is denoted by $ \ell^\eta_{K}$.
When $K = \mathbb{Z}^d$ we simply use $ \ell^\eta$.

\subsection{Branching IDLA} \label{sec:BIDLA}
We give here a formal definition of the BIDLA process using the formalism of particle systems. This definition may, at first sight, look different from that presented in the introduction. The equivalence between the two definitions follows from the Abelian property, which, roughly speaking, gives some freedom on how to stabilise particle configurations without altering their statistics. The fact that classical IDLA enjoys the Abelian property is well known \cite{diaconis1991growth,levine2018long}. We give a detailed proof of the Abelian property for BIDLA in Appendix \ref{app:Abelian}, which may be skipped upon first reading. 

A BIDLA process on $\mathbb{Z}^d$  is defined as follows. 
A particle configuration on $\mathbb Z^d$ is a map $\eta : \mathbb Z^d \to \mathbb N$, where $\eta (z)$ denotes the number of particles at $z\in \mathbb Z^d$. Particles are indistinguishable. 
We build the aggregate through sequential {\it toppling} of {\it unstable} sites:
we say that the site $z$ is \emph{unstable} for the configuration $\eta$  if $\eta (z) > 1$, and
that $\eta$ is \emph{stable} if for any $z\in \mathbb Z^d$, $\eta(z)\le 1$ (at most one particle per site). Choose an arbitrary ordering of the sites of $\mathbb Z^d$, for instance lexicographical ordering.
If a particle configuration $\eta$ is unstable, we define its \emph{BIDLA stabilization} as follows. 
In $\eta$, take the first unstable site, and perform a \emph{toppling}:  
one particle at that site dies after producing a random number of children with law $\nu$ at the same site and, instantaneously, each child takes an independent simple random walk step, thus moving to a random neighbouring site.  
Since the offspring law $\nu$ is critical, iterating this procedure will make the particle configuration stable after finitely many topplings, almost surely. We denote this stable configuration by $\mathscr{S} (\eta )$, and we call it the BIDLA stabilization of the configuration $\eta$.

We consider the evolution of an interacting particle system made of critical Branching Random Walks.
Each particle has two states: {\it settled} or {\it active} (i.e.\ {\it unsettled}). A settled particle stays in place forever, and there is at most one settled particle per site of $\mathbb Z^d$.
The positions of the settled particles define a finite subset of $\mathbb Z^d$ called the {\it explored region}, and this region evolves until all particles are eventually settled, which happens almost surely.

Each unsettled particle has a {\it critical} branching mechanism: it produces a random number of children (or \emph{offsprings}), and each child independently moves away from its parent's position according to a simple random walk step. 
We call this action a {\it toppling}. More formally, a toppling
consists of replacing a particle at an unstable site $x$ (and unstable means $\eta(x)>1$) by $k$ children with probability $\nu(k)$, where we recall that 
 \begin{equation}\label{def-criticalmu}
  \sum_{k\ge 0} k\cdot  \nu(k)=1
 \end{equation}
by Assumption \eqref{assumption}. 
These $k$ children are given independent simple random walk  increments $(\xi_j,\ j\le k)$, 
and the particle configuration $\eta$ is replaced by 
\[
\eta -\1_{x}+\sum_{j=1}^k \1_{x+\xi_j}.
\]
The dynamics stops when no toppling is possible, that is when  $\eta(x)\le 1$ for all $x\in \mathbb Z^d$, and the resulting stable configuration is denoted by $ \mathscr{S}(\eta)$. 

With this notation in place, a BIDLA process $(A(t))_{t\in \mathbb N }$ is defined inductively by setting $A(0) = \emptyset$ and, for each $t\geq 0$, 
	\[ A(t+1) := \mathscr{S} ( A(t) + \1_0 ) . \]
In words, at each discrete time step $t\geq 1$ add a particle  at the origin and perform BIDLA stabilization. 

\begin{Remark}
We remark that the definition of BIDLA presented in the introduction simply corresponds to choosing a particular order of the topplings, the one dictated from the exponential clocks associated to the BRW particles. By the Abelian property, which is spelled out in Appendix \ref{app:Abelian}, the law of the stable configuration does not depend on the order of the topplings, thus the two definitions of BIDLA are equivalent. 
\end{Remark}

Note that, since $\nu $ is assumed to be critical, 
the total number of particles is a martingale. Moreover, the stabilization time is a stopping time in the natural filtration associated to each driving Branching Random Walk (or, alternatively, to each toppling). Thus, recalling that $|\eta |$ denotes the number of particles in $\eta$, we have the following. 
\begin{Lemma}\label{lem-growth}
For any dimension $d\geq 1$ and any finite particle configuration $\eta$ on $\mathbb{Z}^d$, we have that
\begin{equation}\label{eq-growth}
\bE[|\sS(\eta)|]=|\eta|.
\end{equation}
\end{Lemma}

\subsection{Stabilization in a finite domain} \label{sec:domain} 
As for classical IDLA, when working with BIDLA it is often convenient to \emph{partially stabilize} a particle  configuration $\eta$ inside a given domain $K \subseteq \mathbb{Z}^d$. This means that we only topple sites in $K$, and when a particles steps outside of $K$ we declare it {\it frozen}, thus pausing its BRW evolution. Almost surely, after finitely many topplings this procedure will result in a stable configuration inside $K$, together with a collection of frozen particles on the exterior boundary $\partial K$. We denote this configuration by $\mathscr{S}^\eta_K$. Hence by definition $\mathscr{S}^\eta_K (z) \leq 1$ for all $z \in K$, while if $z \in \partial K $ then $\mathscr{S}^\eta_K (z) \in \mathbb{N}$ denotes the number of frozen particles at $z$. These particles are on standby, and their evolution may be continued at a later time. Importantly,  by the Abelian property
	\[ \cS_{\mathbb Z^d}^{\eta}\stackrel{\text{def}}{=} 
	\mathscr{S} ( \eta ) 
	\stackrel{\text{law}}{=} \mathscr{S} ( \cS_K^{\eta} ).
	\]
In words, freezing particles and releasing them at a later time will result in a stable configuration which has the same statistics as the one obtained by stabilising without freezing. See Appendix \ref{app:Abelian} for a detailed discussion on the Abelian property.

\section{Preliminaries}\label{sec:prelim}
We collect in this section some preliminary results on BRWs and Green's function estimates, which will be used throughout the paper. From here onwards,  we denote by $c,C$ constants which may depend only on the dimension $d$, and which are allowed to change from line to line. 
\subsection{On Branching Random Walks}
Let $K \subset \mathbb{Z}^d$ be a finite subset. 
Recall from Section \ref{sec:restricted} that if $y \in  K$ then $ \ell^y_K (z)$ denotes the local time at $z$ of a BRW starting at $y$ with freezing upon exiting $K$. Then 
\begin{equation}\label{eq:pioneers}
	\ell_K^y(\partial K) := \sum\limits_{z \in \partial K} \ell_K^{ y } (z)
\end{equation}
counts the number of frozen particles on $\partial K$. We call these particles \emph{pioneers}, as they are the first individuals in the BRW process to explore the space outside $K$. 

The first result of this section tells us that the probability that a BRW survives to distance $R$ from its starting location is of order $\Theta ( 1/R^2 )$, independently of the dimension.

\begin{Proposition}[\cite{asselah2022time}, Lemma 3.8] 
\label{le:distance}
For any  $d \geq 1$ there exist constants $c, C > 0$, only depending on the dimension $d$, such that for any $R \geq 1$,
\[\begin{split} 
& 
\inf_{y \in B_R} \mathbb{P} \big(  \ell^y_{B_R}(\partial B_R)  > 0 \big) \geq \frac{c}{R^2}, 
\\ 
& 
\sup_{y \in B_{R/2}}  \mathbb{P}( \ell^y_{B_R}(\partial B_R) > 0) \leq \frac{C}{R^2}.
\end{split} 
\]
\end{Proposition}
Conditional on survival up to $\partial B_R$, one may ask how many pioneers are frozen on $\partial B_R$. 

\begin{Proposition}[\cite{asselah2022time}, Theorem 1.3]
\label{prop:largedeviation}
For any $d\geq 1$ there exist positive constants $c $ and $\lambda_0$, depending only on the dimension $d$, such that for each $ \lambda \in (0,  \lambda_0]$ 
and for each $y \in B_R$
\begin{align}\label{eq:Amines estimate}
 \mathbb{E} \Big[  \exp\Big(  \lambda  \ell_{B_R}^y ( \partial B_R)  /R^2 \Big)  \Big]
  &  \leq e^{  c  \lambda / R^2 } \\
\label{eq:Amines estimate conditional}
  \mathbb{E} \Big[\exp \Big(   \lambda  \ell_{B_R}^y ( \partial B_R) / R^2 \Big)  \Big| \ell_{B_R}^y ( \partial B_R) > 0 \Big] &  \leq e^{  c  \lambda }. 
\end{align}
Moreover, writing $ \bar \ell_{B_R}^y := \ell_{B_R}^y ( \partial B_R) - \E [ \ell_{B_R}^y (\partial B_R) ] $, we have that 
\begin{equation}\label{eq:Amine_centered}
  \mathbb{E} \Big[\exp \Big(   \lambda \bar \ell_{B_R}^y  / R^2 \Big)  \Big| \ell_{B_R}^y ( \partial B_R) > 0 \Big]  \leq e^{  c  \lambda^2 / R^2 }. 
\end{equation} 
\end{Proposition}
Thus, on the event that the BRW survives to distance $R$,  the number of particles
reaching the boundary of $B_R$  is of order $\Theta (R^2)$. 
As a direct corollary of this result we get the following. 

\begin{Corollary}\label{cor:conditioned}
For any $d\geq 1$ there exists $\bar \beta (d) \in (0, \infty ) $ such that the following holds. For any $\alpha > 0$ and any $\beta \geq \bar \beta (d) $ there exists a positive constant $c = c(\alpha , \beta ,d)$ such that for all $R>0 $ 
	\[ \P \big( \ell^0_{B_R} ( \partial B_R )  \in [\alpha R^2 , \beta R^2 ]
	\big| \ell^0_{B_R} ( \partial B_R ) >0 \big) \geq c(\alpha , \beta ,d) >0 . \]
\end{Corollary}
\begin{proof}
Write $\ell_R $ in place of $\ell^0_{B_R} ( \partial B_R ) $ for brevity. We use the conditional version of the Paley-Zygmund inequality: for any non-negative square-integrable random variable $X$, any event $E$ with $\P (E)>0$ and any $\vartheta \in [0,1]$,  it holds 
	\[ \P \big( X \geq \vartheta \E [X|E] |E \big) \geq (1-\vartheta )^2 \frac{(\E[X|E])^2}{\E [X^2 |E ]} . \]
Taking $X = \ell_R $ and $E = \{ \ell_R >0\}$ the above inequality rewrites as 
	\[ \P \big( \ell_R \geq \vartheta  \cdot \E[\ell_R | \ell_R >0 ] \, \big| \ell_R >0 \big) 
	\geq ( 1- \vartheta )^2 \frac{ (\E[\ell_R | \ell_R >0 ])^2 }{\E [ \ell_R^2 | \ell_R >0 ] } .  
	\]
Now by Lemma 3.8 of \cite{asselah2022time} we know that there exist constants $C_1 , C_2 \in (0,\infty ) $, depending only on the dimension $d$,  such that 
	\[ \E [ \ell_R | \ell_R >0 ] \geq C_1 R^2 , 
	\qquad 
	\E [ \ell_R^2 | \ell_R >0 ] \leq C_2 R^4 , 
	\]
which gives 
	\[  \P \big( \ell_R \geq \vartheta \E [\ell_R | \ell_R >0 ] \, \big| \ell_R >0 \big) 
	\geq
	(1 - \vartheta )^2 \frac{C_1^2}{C_2 } . \]
Since, moreover, 
	\[  \E [ \ell_R | \ell_R >0 ] \leq C_3 R^2 , \]
we conclude that 
	\[ \P \big( \ell_R \geq C_3 \vartheta R^2 \big| \ell_R >0 \big) 
	\geq (1 - \vartheta )^2 \frac{C_1^2}{C_2 } \]
for any $\vartheta \in [0,1]$. Take $\vartheta = \bar \vartheta (\alpha ) :=  \alpha / (2 \max \{ C_3 , \alpha \} ) \in (0,1/2]$ to get that 
	\[ \P \big( \ell_R \geq \alpha  R^2 \big| \ell_R >0 \big) 
	\geq (1 - \bar \vartheta (\alpha )  )^2 \frac{C_1^2}{C_2 } >0 . \]
It remains to show that $\P ( \ell_R \geq \beta R^2 | \ell_R >0 ) $ can be made arbitrarily small by taking $\beta$ large enough. Indeed, recall that by Proposition  \ref{prop:largedeviation} we have that for any $d\geq 1$ there exist constants $c=c(d) ,\bar \lambda = \bar \lambda (d)>0$ such that for any $\lambda \in (0, \bar \lambda )$ it holds that  
	\[ \E [ e^{\lambda \ell_R / R^2} | \ell_R >0 ] \leq e^{c\lambda } . 
	\]
Taking $\lambda >0$ as above we thus find 
	\[ \begin{split} 
	\P ( \ell_R \geq \beta R^2 | \ell_R >0 ) 
	= \E [ e^{\lambda \ell_R / R^2} | \ell_R >0 ] \cdot e^{-\lambda \beta } 
	\leq e^{- \lambda ( \beta - c ) } ,  
	\end{split} 
	\]
which can be made arbitrarily small by taking $\beta $ large enough. 
\end{proof}

\subsection{Green's function estimates}
\label{sec:Green}
For $x \in \mathbb{Z}^d$,  denote by $(S^x_n )_{n\in \mathbb N}$ a simple random walk on $\mathbb{Z}^d$ starting from  $S^x_0=x$. 
Given a finite set $K  \subset \mathbb{Z}^d$, let 
	\[ H_K(S^x)  := \inf\{ t\geq 0    :  \,  S^x_t\in K\} \]
denote the hitting time of $K$, and write simply $H_K$ when no confusion occurs on the random walk.
For any $x,y \in \mathbb{Z}^d$, define the Green's function by 
$$
G_K(x,y):= \bE \left [ \sum\limits_{n = 0}^{H_{K^c} } \1 (S^x_n = y ) \right],
$$
where $K^c := \mathbb{Z}^d\bs K$.
Recall that $B_R$ denotes the ball of radius $R$ centered at the origin. We use the notation $G_R(x,y)$ in place of  $G_{B_R}(x,y)$.
Note that, with our definition, when $y\in \partial B_R$
we have that $G_R(x,y) =\bP\big( S^x(H_{\partial B_R})=y\big)$.

We start by recalling a classical bound of Asselah and Gaudilli\`ere. 
\begin{Lemma}[\cite{asselah2013logarithmic}, Lemma 5.1]\label{lem-AG}
Assume $d\ge 2$. Then there is an absolute constant $\kappa$ such that  for any $x\in B_R$ and $y\in \partial B_R$ it holds that
\begin{equation}\label{AG-hitting}
\bP\big( S^x(H_{\partial B_R})=y \big)\le \kappa \cdot \frac{1}{\|x-y\|^{d-1}}.
\end{equation}
\end{Lemma}
In dimension $d \geq 3$ we need the following natural and crucial generalisation, which allows us to take $y$ anywhere inside $B_R$. Write $a \wedge b $ for the minimum between $a$ and $b$. 
\begin{Lemma}\label{lem-rasemote} Assume that $d \geq 3$. 
For arbitrary  $R>0$ let  $x\in B_R$ and $y\in B_R \cup\partial B_R$. Then there is an absolute constant $\kappa$ such that 
\begin{equation}\label{ineq-rasemote}
\bP\big( H_y(S^x)\le H_{\partial B_R} (S^x) \big)\le \kappa \cdot \frac{1}{\|x-y\|^{d-2}}
 \cdot \Big(1\wedge \frac{(R-\|y\|+1)}{\|x-y\|}\Big)\cdot \Big(1\wedge \frac{(R-\|x\|+1)}{\|x-y\|}\Big).
\end{equation}
\end{Lemma}
Since we have not found a reference for this result  in the extensive literature on random walks, we choose to include the proof. 
\begin{proof} 
Define $R_x:=R-\|x\|$ and $R_y:=R-\|y\|$. 
Assume first that $\max(R_x,R_y)\le r$ where we set $r:=\|x-y\| /16$. 
In other words, the points $x,y$ are both much closer
to the boundary than to each other. In this case, we show that two simple random walks starting from $x$ and $y$ respectively must first get away from the
boundary by a distance of order $\|x-y\|$ before they can meet. 
To see this, define two points $x',y'\in \mathbb Z^d$ outside $B_R$ as follows.
 Take $x'$ at a distance $1$ from the line $\mathbb R\cdot x$ and such that $x'\in \partial B_{R+2r}$ and 
$B_R\cap (x'+B_{2r})=\emptyset$. The same construction holds for $y'$ when 
the role of $x$ and $y$ are exchanged.
Then $x$ is a distance less than $R_x+1$ from $x'+B_{2r}$, 
$y$ is a distance less than $R_y+1$ from $y'+B_{2r}$,  and 
\begin{equation}\label{opt-green1}
\|x'-y'\|\ge \|x-y\|-r=15r.
\end{equation}
The point of this construction is that a path from $x$ to $y$ avoiding the complement of $B_R$ must first reach $B_1:=x'+B_{4r}$ before $x'+B_{2r}$, then reach $B_2:=y'+B_{4r}$,  
then go from $\partial B_2$ to $y$ without hitting $y'+B_{2r}$. 
\begin{figure}[!!ht]
    \begin{center}
     \centering
  \mbox{\hbox{
  \hspace{7mm} \includegraphics[width=.55\textwidth]{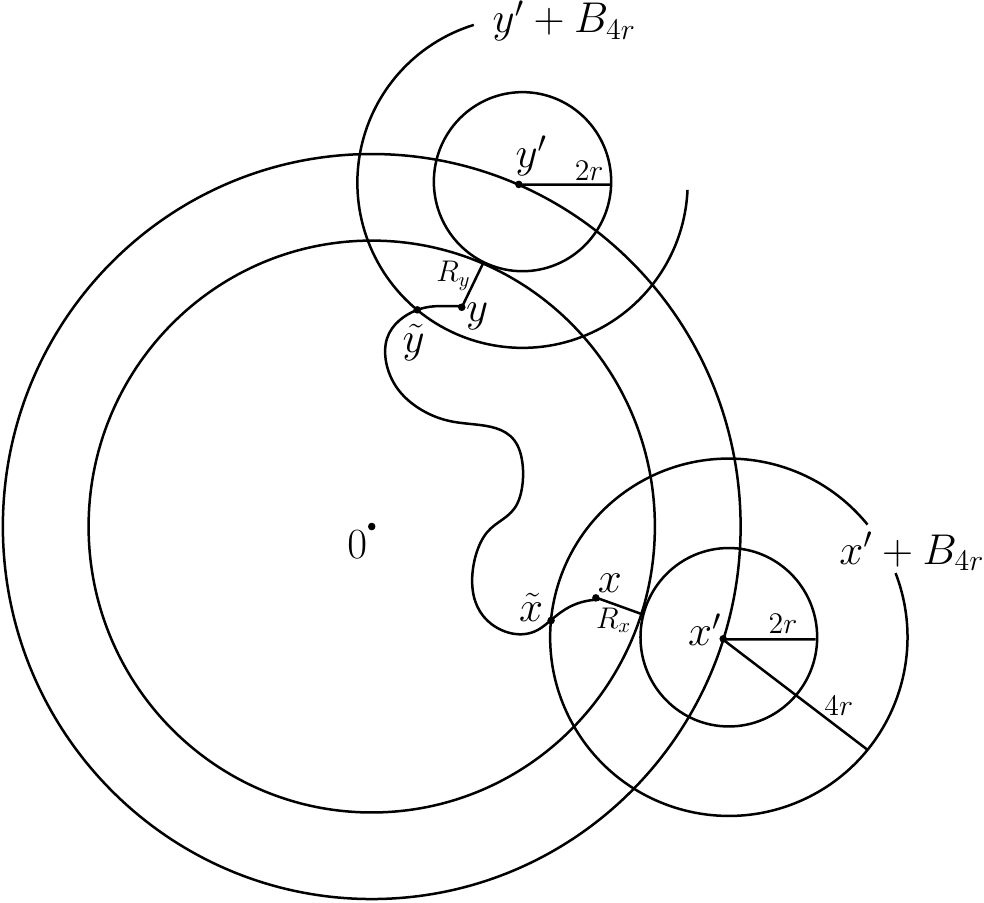}}}
       \end{center}
 \caption{A depiction of the construction used in the proof.} \label{fig:LemmaGreen}
 \end{figure}

Thus, 
	\[ \{H_y(S^x)<  H_{\partial B_R} (S^x) \big\}
	\subseteq  \{H_y(S^x)<  \min(H_{x'+B_{2r}},H_{y'+B_{2r}})\} , \]
where we have suppressed the random walk from the notation in the right hand side for brevity. 
Now if we call $\tilde x$ the
first site the walk $S^x$ hits on $\partial B_1$, and $\tilde y$ the last site it hits on
$\partial B_2$ before $y$, we can  partition the event
$\{H_y(S^x)<   \min(H_{x'+B_{2r}},H_{y'+B_{2r}})\big\}$ according to the values of $\tilde x $ and $\tilde y$, to get that its probability is bounded by
\begin{equation}\label{opt-green2}
\begin{split}
 \sum_{\tilde x,\tilde y} &
\bP\big( S^x(H_{\partial B_1})=\tilde x, H_{\partial B_1}<H_{x'+B_{2r}}\big)\cdot \bP\big( H_{\tilde y} (S^{\tilde x})<\infty\big)
\cdot \bP\big( S^y(H_{\partial B_2})=\tilde y, H_{\partial B_2}<H_{y'+B_{2r}}\big)\\
 & \leq \, \bP\big( H_{\partial B_1} (S^x)<H_{x'+B_{2r}}\big)\cdot 
\sup_{\tilde x \in  \partial B_1 , \, \tilde y \in \partial B_2}
\bP\big( H_{\tilde y} (S^{\tilde x})<\infty\big)\cdot 
\bP\big( H_{\partial B_2} (S^{y})<H_{y'+B_{2r}}\big).
\end{split}
\end{equation}
Here we used the strong Markov property and the reversibility of the random walk to obtain the second
line in \reff{opt-green2}.
Now, by a Gambler's ruin estimate, there exists $\kappa>1$ such that
\begin{equation}\label{opt-green3}
\begin{split} 
& \bP\big( H_{\partial B_1} (S^x)<H_{x'+B_{2r}}\big)\le 
\kappa\cdot \frac{R_x+1}{\|x-y\|},\quad\text{and}\\ & 
\bP\big( H_{\partial B_2} (S^y)<H_{y'+B_{2r}}\big)\le \kappa\cdot \frac{R_y+1}{\|x-y\|} . 
\end{split} 
\end{equation}
Combining this with  the classical estimate $\bP\big( H_{\tilde y} (S^{\tilde x})<\infty\big)\le 
\kappa\cdot \|\tilde x-\tilde y\|^{2-d}$ and the fact that $\|\tilde x-\tilde y\|> 7r$, we gather that  for some constant $C>0$
\begin{equation}\label{opt-green4}
\bP\big( H_y(S^x)<H_{\partial B_R} (S^x) \big)\le C  \frac{R_x+1}{\|x-y\|}\cdot \frac{R_y+1}{\|x-y\|}\cdot \frac{1}{ \|x-y\|^{d-2}}.
\end{equation}
This proves \eqref{ineq-rasemote} in the case $\max (R_x, R_y) \leq r$. 
It remains to treat the other cases. 
If  $R_x\le \|x-y\|/16 < R_y$, the only difference with the previous construction
is that once we reach $\partial B_1$, at $\tilde x$, we only require the walk to hit $y$, and this yields
the result. The case  $R_y\le \|x-y\| /16 < R_x$ is handled analogously,  exchanging the role of $x$ and $y$.
Finally, if $\min(R_x,R_y)> \|x-y\| /16$, then this means 
that $x$ and $y$ are in the bulk of $B_R$, and
the unrestricted bound holds.
\end{proof} 
In the remainder of this section we focus on obtaining  upper bounds on the second moment of the local time of branching random walks restricted to a  Euclidean ball. 
Recall from Section \ref{sec:BRW} that if $\cT$ is a critical Galton-Watson tree indexing a BRW,  for $R>0$ and $x\in B_R$ we denote by $\cT^x(B_R)$ the set of vertices of $\cT$ whose ancestors are all mapped to $B_R$, when the root is mapped to $x$. 
The vertices of $\cT^x(B_R)$ mapped to $\partial B_R$ are the pioneers.
When there is no confusion, we denote by $\cT^x(B_R)$ both the set of vertices and their positions in $\mathbb{Z}^d$, with a slight abuse of notation. Recall that the local time at site $z\in B_R$ of a BRW $( S^x_u )_{u \in \cT^x (B_R)} $ starting at $x$ is defined by 
\[
\ell^x_{B_R}(z)=\sum_{u\in \cT^x(B_R)} \1  ( S^x_u=z ), \]
and 
\[
G_R(x,z) =\E [\ell^x_{B_R}(z)].
\]
When $z\in \partial B_R$, the local time $\ell^x_{B_R}(z)$ counts the number of pioneers which exit $B_R$ at $z$, and 
$G_R(x,z) =\P\big(S^x(H_{\partial B_R})=z\big)$, so that  $\E[\ell^x_{B_R}(z)]=G_R(x,z).$

The following lemma bounds the second moment of the local time of a BRW restricted to $B_R$ in all dimensions. Recall that $\sigma^2$ denotes the variance of the offspring distribution, which is positive by Assumption \eqref{assumption}. 
\begin{Lemma}\label{lem-l2}
For any $d \geq 1$,  any $R>0$, $x\in B_R$ and $z \in \partial B_R$ it holds 
\begin{equation}\label{green-local2}
\E \big[ \big(\ell^x_{B_R}(z)\big)^2 \big] \leq  G_R(x,z)+  \sigma^2 \sum_{y\in B_R} G_R(x,y)\cdot G_R^2(y,z).
\end{equation}
\end{Lemma}
\begin{proof} 
As we expand the square, the second moment corresponds to summing over $u,u'\in  \cT^x$, and it is convenient to consider the youngest common ancestor, say $v$ of $u$ and $u'$ when they are different. Thus,
if $\cT^v$ is the subtree of $\cT^x (B_R)$ having $v$ as its root and including only the descendants of $v$
\begin{equation}\label{green-1}
\begin{split}
\big(\ell^x_{B_R}(z)\big)^2=&\ell^x_{B_R}(z)+\sum_{u\not= u'\in \cT^x(B_R)}\1 (S^u=z=S^{u'} )\\
=&\ell^x_{B_R}(z)+\sum_{v\in \cT^x(B_R)}\sum_{u\not= u'\in \cT^v}\1 ( S^u=z=S^{u'})\\
=  &\ell^x_{B_R}(z)+ \sum_{y\in B_R}  \sum_{v\in \cT^x(B_R)}\1 (S^v=y ) 
\sum_{u\not= u'\in \cT^v}\1 (S^u=z=S^{u'}) . 
\end{split}
\end{equation}
Now, the tree $\cT^v$ when $S^v=y$ has the same law as $\cT^y(B_R)$. 
Moreover,   if $X_v$ denotes the number of children of $v$, then  $\E [{X \choose 2} ] =  \sigma^2 / 2 <\infty$. Taking expectations throughout \eqref{green-1}  we obtain (\ref{green-local2}).
\end{proof}
\begin{Remark}
The same computation shows that if $z \in B_R$ then 
	\[ \E \big[ \big(\ell^x_{B_R}(z)\big)^2 \big] 
	\leq  G_R(x,z) \cdot ( 1 + 2G_R^+ (z,z) ) +  \sigma^2 \sum_{y\in B_R} G_R(x,y)\cdot G_R^2(y,z) , 
	\]
where $G_R^+ (z,z) $ counts the expected number of returns to $z$ before exiting $B_R$. 
\end{Remark}
We now obtain upper bounds for the right hand side of \eqref{green-local2}. 
Recall that
\[
\bP(H_y(S^x)<H_{\partial B_R} \big)=\frac{G_R(x,y)}{G_R(y,y)}.
\]
Moreover, 
$G_R(y,y)$ is uniformly bounded in the transient case, whereas in the recurrent case $G_R(y,y)$ can reach a value of order $\log R$, so extra care is needed.

We start by providing an estimate in the transient case. 
\begin{Lemma}\label{lem-2moment}
Assume $d \geq 3 $.
There are finite constants $\{C_d,d \geq 3\}$ such that for any $R>0$, $x\in B_R$, and $z\in \partial B_R$ it holds 
\begin{equation} \label{Cd}
\sum_{y\in B_R} G_R(x,y)G_R^2(y,z)\le  \frac{R-\|x\|+1}{\|x-z\|^{d}}
\times \begin{cases}
C_3\log \|x-z\| & \mbox{ if $d=3$,} \\
C_d & \mbox{ if $d \geq 4$.} 
\end{cases}
\end{equation}
\end{Lemma}
\begin{proof}
Since 
$ G_R(x,y)\le G_R (0,0)\cdot \P (H_y(S^x)<H_{\partial B_R})$, 
and $G_R(0,0)$ is uniformly bounded in the transient case, we have that $G_R(x,y)$ is uniformly bounded for all $x ,y\in B_R$. It thus suffices to bound 
$G_R(y,z)$ for arbitrary $y \in B_R$ and $z \in \partial B_R$. This can be achieved using  Lemma~\ref{lem-rasemote}. 

We split the sum over $y \in B_R$ into three regions such that $B_R = \mathcal{R}_1\cup \mathcal{R}_2 \cup  \mathcal{R}_3 $, defined as follows. 
\begin{itemize}
\item Region 1 corresponds to $y$ being far from 
$z$, i.e.\ 
$ \mathcal{R}_1 = \{y:\ \|y-z\|> 2\|x-z\|\} . $
\item Region 2 corresponds to 
$y$ being close to $x$, i.e.\ 
$ \mathcal{R}_2 = \Big\{ y:\ \|y-x\|< \frac{1}{2}\|x-z\|\Big\} . $
\item Region 3 corresponds to the remaining space in $B_R$, i.e.\  
$  \mathcal{R}_3 = B_R \setminus (\mathcal{R}_1\cup \mathcal{R}_2 ) $. 
\end{itemize}
Recall the notation $R_x=R-\|x\|\le \|x-z\|$.

In Region 1, we bound $G_R(x,y)$ by $\kappa \cdot R_x\cdot \|x-y\|$, 
and $G_R(y,z)$ by $c\|x-y\|^{1-d}$, for some finite constants $\kappa ,c$. This gives the bound 
 \[
 \sum_{y\in \mathcal{R}_1 } 
 \frac{R_x}{\|x-y\|^{d-1}} \cdot \frac{1}{\|x-y\|^{2d-2}}
 \le c \cdot   \frac{R_x}{\|x-z\|^{2d-3}}
 \le
 \frac{c}{\|x-z\|^{d-3}}\cdot \frac{R_x}{\|x-z\|^{d}}.
 \]
In Region 2, $G_R(y,z)=\mathcal{O} (\|x-z\|^{1-d})$, and the sum over $G_R(x,y)$ in this region
 is of order $R_x\cdot \|x-z\|$. 
 This gives the bound 
 \[
  \sum_{y\in \mathcal{R}_2 }  \frac{R_x}{\|x-y\|^{d-1}} \cdot \frac{1}{\|x-z\|^{2d-2}}
  \le c \cdot   \frac{R_x}{\|x-z\|^{d-1}} \cdot  \frac{1}{\|x-z\|^{d-2}}.
  \]
We are left with Region 3. For $h \leq 2\|x-z\| $ and $k \leq 2\|x-z\|/h$,  define  the shells 
$\Sigma^z_{h,k}:=(z+B_{(k+1)h}\bs B_{kh})\cap \{y:\ h\ge R_y >h-1\}$.
 Note that $|\Sigma^z_{h,k}|\le c h\cdot(kh)^{d-2}$. 
 Thus, since in this region $\|x-y\|\ge \frac{1}{2}\|x-z\|$, we end up with 
 \begin{equation}\label{green-10}
 \begin{split}
\sum_{y\in \mathcal{R}_3} G_R(x,y)G_R^2(y,z)
\le & \, c \sum_{h=1}^{2\|x-z\|}
\sum_{k=1}^{2\|x-z\|/h} \sum_{y\in \Sigma^z_{h,k}} \frac{h\cdot R_x}{\|x-y\|^{d}}
\cdot \frac{1}{\|y-z\|^{2(d-1)}}\\
\le & \, c \cdot \frac{R_x}{\|x-z\|^{d}}\sum_{h=1}^{2\|x-z\|}\frac{1}{h^{d-2}}\sum_{k=1}^{2\|x-z\|/h}  \frac{1}{k^d}
\\  & \leq c \cdot \frac{R_x}{\|x-z\|^{d}} \cdot 
\big( \log \|x-z\| \1 _{d=3} + \1_{d\geq 4} \big) . 
\end{split}
\end{equation}
Gathering the above estimates concludes the proof.
\end{proof}

Plugging this bound into \eqref{Cd}  we obtain the following corollary.
\begin{Corollary}\label{cor:localtimeboundaryball}
Let $d \geq 3$. 
There exists $c _d< \infty$, depending only on the dimension, 
such that for each large enough $R$ and any $x \in B_R$ we have that 
\begin{equation}\label{2dmoment-boundary}
\sum_{z\in \partial B_R} \sum_{y\in B_R} G_R(x,y)G_R^2(y,z) \leq  
c_d \times \begin{cases}
\log( R - \|x\|  )  &  \mbox{ if $d=3$,} \\
1  & \mbox{ if $d \geq 4$.} 
\end{cases}
\end{equation}
\end{Corollary}

We now move to the recurrent case, where the following holds. 
\begin{Lemma}\label{lem-2m-recurrent}
Assume $d=2$.
Then there exists a finite absolute constant $C_2$ 
such that for any $R>0$ and any $z\in \partial B_R$
\begin{equation}\label{green-4}
\sum_{y\in B_R} G_R(0,y)G_R^2(y,z)\le C_2.
\end{equation}
\end{Lemma}
\begin{proof}
To start with, note that if $y \in B_R$ then $G_R(0,y)\le \kappa\cdot \log(R/\|y\|)$. If, moreover,  $z\in \partial B_R$, by Lemma~\ref{lem-AG} we have that 
\[
G_R (y,z) = \bP(H_{\partial B_R}(S^y)=z)\le\frac{\kappa }{\|y-z\|} . 
\]
This implies that for any integer $h\le R$
\[
\sum_{y:\ h-1\le  \|y\|< h} G_R^2 (y,z ) 
\le \frac{ c}{R-h+1}
\]
for some $c$ absolute constant. 
It follows that 
\[
\sum_{y\in B_R} G_R(0,y)G_R^2(y,z) 
\leq c\kappa\cdot \sum_{h=1}^R \log(R/h)\cdot \frac{1}{R-h+1}.
\]
By a Riemann integral approximation,  
\[
\lim_{R\to\infty} \sum_{h=1}^R \log(R/h)\cdot \frac{1}{R-h+1}=\int_0^1 \frac{\log(1/x)}{1-x}\ dx<\infty , 
\]
from which \eqref{green-4} follows. 
\end{proof}

To conclude, we recall a discrete harmonic property for the restricted Green's function. 
\begin{Lemma}\label{lem-Harm} Let $d>2$,  $R>0$ and $z\in \partial B_R$. There is a positive constant $C$
independent of $R$ and $z$ such that
\begin{equation}\label{green-7}
\Big| |B_R| G_R(0,z)-\sum_{y\in B_R} G_R(y,z) \Big|\le C.
\end{equation}
\end{Lemma}
\begin{proof}
This is a simple consequence of Theorem 5.2 of \cite{asselah2013sublogarithmic}, and Proposition B.1 of the Appendix of \cite{asselah2013sublogarithmic}.
Indeed, Theorem 5.2 of \cite{asselah2013sublogarithmic}  establishes that for some positive constants $c ,c '$ (independent of $R$)
\begin{equation}\label{green-8}
\Big| |B_{R-c }|\cdot G_R(0,z)-\sum_{y\in B_{R-c }} G_R(y,z) \Big| \le c ' .
\end{equation}
If we define $\mathfrak D:=|B_R| G_R(0,z)-\sum_{y\in B_R} G_R(y,z)$, then 
\[
\begin{split}
|\mathfrak D |\le &\Big(|B_R|-|B_{R-c }|\Big)\cdot G_R(0,z) +\Big|
|B_{R-c }|\cdot G_R(0,z)-\sum_{y\in B_{R- c}} G_R(y,z)\Big|\\
+& \sum_{y\in B_R\bs B_{R- c }} G_R(y,z) .
\end{split}
\]
Now, $(|B_R|-|B_{R- c }|)G_R(0,z)$ is of constant order, and the sum over the annulus
$B_R\bs B_{R- c }$ of $G_R(y,z)$ can be thought of as the time spent in the annulus by a simple random walk 
before exiting $B_R$. Proposition B.1 of the Appendix of \cite{asselah2013sublogarithmic} then establishes that this is of constant order, which concludes the proof. 
\end{proof}

\subsection{A classical approach to the inner bound} \label{sec:classical}
Here we recall a classical approach to proving the inner bound of a shape theorem, introduced by Lawler, Bramson and Griffeath in the seminal paper \cite{lawler1992internal}, together with a key independence feature observed by Asselah and Gaudilli{\`ere} in \cite{asselah2013logarithmic}. 
This approach will be used in Section \ref{sec:d2} to prove a weak inner bound for BIDLA  in dimension $d=2$, and in Section \ref{sec:innerbound} to prove a strong inner bound for $d\geq 3$.

Let $\eta$ be an arbitrary initial configuration, and recall that $|\eta|=\sum_{z\in \mathbb{Z}^d} \eta(z)$ denotes the number of particles in $\eta$. With a slight abuse of notation, we think of $\eta$ also as a set of positions, and write $y\in \eta$ when summing over  the positions of the $|\eta|$ sites in $\eta$.

Assume that $\eta $ is supported in a Euclidean ball $B_R$ for some $R>0$. 
We want to show that stabilising $\eta$ with freezing on $\partial B_R$ results in filling the ball $B_R$ with high probability. To this end we need upper bounds for 
	\[ \P ( B_R \nsubseteq \mathscr{S}_{B_R}^{\eta} ) 
	\leq \sum_{z\in B_R} \P ( \mathscr{S}_{B_R}^{\eta} (z) =0 ) . \]
The idea proposed by Lawler, Bramson and Griffeath is the following: let first each particle in $\eta$ evolve as a BRW with settling upon reaching an empty site and freezing upon reaching $\partial B_R$. 
The settling locations inside the ball are clearly correlated, since whether a site is empty or not at a given time depends on the trajectories of previous particles. However, if we \emph{continue} the evolution of the BRWs \emph{after settling}, which is sometimes referred to as introducing \emph{ghost particles}, we can complete the particles' trajectories to full \emph{independent} BRWs. 
Now, while the starting locations of the ghosts are correlated, there is at most one ghost per site inside the ball $B_R$. Thus we can only increase the chance of reaching a boundary site $z$ if we release \emph{exactly} one ghost per site. This allows one to recover independence. 

To formalise this strategy, fix an arbitrary site $z \in B_R \cup \partial B_R$, and note that  $\cS^\eta_{B_R}(z)$ is the local time at $z$ in the BIDLA phase, i.e.\ the number of visits at $z$ when releasing one particle from each site of $\eta$ and letting them evolve as in BIDLA dynamics with freezing upon reaching $\partial B_R$. 
If $z \in \partial B_R$ then $\cS^\eta_{B_R}(z)$ denotes the number of particles frozen at $z$. 

Write $\eta_G := \cS^\eta_{B_R} \cap B_R$ for the particle configuration obtained by BIDLA-stabilising $\eta$ inside $B_R$. Then 
	\begin{equation}\label{eq:centralequality}
	\cS^\eta_{B_R}(z)+\tilde \ell^{\eta_G }_{B_R}(z) = \ell^{\eta}_{B_R}(z) , 
	\end{equation}
where  $\tilde \ell^{\eta_G}_{B_R}(z)$ denotes the local time of the \emph{ghosts}, that is the number of visits to $z$ when releasing exactly one BRW from each site of $\eta_G$ with freezing on $\partial B_R$. The tilde notation is there to emphasize that, given their starting locations, the evolution of the ghosts is \emph{independent} from that of the BIDLA particles. Now, since there is at most one ghost per site, 
	\[ \tilde \ell^{\eta_G }_{B_R}(z) \leq \tilde \ell^{B_R }_{B_R}(z) , \]
where $B_R$ denotes the particle configuration made of exactly one particle at each site of $B_R$. This has the advantage that the dependence on the configuration $\eta_G$, and hence on the BIDLA particles, is lost. Therefore we get 
  the crucial inequality\footnote{Here \eqref{eq:centralequality} and \eqref{eq:centralinequality} are  analogous to, respectively, the identity $N+L=M$ and the inequality $N+\tilde L \geq M$ in \cite{lawler1992internal}, Section 3.
  }
 	\begin{equation}\label{eq:centralinequality}
	 \cS^\eta_{B_R}(z)+\tilde \ell^{B_R}_{B_R}(z)\geq\ell^{\eta}_{B_R}(z) , 
	 \end{equation} 
with the left hand side being a sum of \emph{independent} random variables. 
It follows that, for any $\lambda >0$, 
\begin{equation}\label{eq:Nzdim2}
\begin{split}
\mathbb{P}(\cS^\eta_{B_R}(z)=0)  &
\leq 
\frac{ \mathbb{E} \big[  e^{- \lambda \cS^\eta_{B_R}(z)} ]\cdot \mathbb{E} 
\big[ e^{- \lambda \tilde \ell^{B_R}_{B_R}(z)} \big]} 
{ \mathbb{E} \big[ e^{- \lambda \tilde \ell^{B_R}_{B_R}(z)}  \big] }  
\leq 
 \frac{ \mathbb{E} \big[ e^{- \lambda \ell^{\eta}_{B_R}(z) }  \big]  }
 {    \mathbb{E}  \big[ e^{- \lambda \tilde \ell^{B_R}_{B_R}(z)} \big]   } 
 \\& =  
 \frac{ \mathbb{E} \big[ e^{- \lambda \bar\ell^{\eta}_{B_R}(z)} \big]     }
 {    \mathbb{E}  \big[ e^{- \lambda\bar \ell^{B_R}_{B_R}(z)}\big ]  }
\cdot  \exp\Big(- \lambda  \mathbb{E}\big[\ell^{\eta}_{B_R}(z)-\ell^{B_R}_{B_R}(z)\big]  \Big) 
 \\ & \leq  
\bigg(  \prod_{y \in \eta } 
 \mathbb{E}   \big[ e^{  - \lambda \bar \ell^{y}_{B_R}(z)  } \big ] \bigg) 
 \cdot 
 \exp\Big(- \lambda \mathbb{E}\big[\ell^{\eta}_{B_R}(z)-\ell^{B_R}_{B_R}(z)\big]  \Big) . 
 \end{split}
\end{equation}
Here we have set  $ \bar \ell^{y}_{B_R}(z) :=  \ell^{y}_{B_R}(z) - \E [  \ell^{y}_{B_R}(z) ] $ for brevity. The second inequality follows from independence of $\cS^\eta_{B_R}(z)$ and $\tilde \ell^{B_R}_{B_R}(z)$, while the last inequality follows from Jensen's inequality. 
Now, using that $e^{ x } \leq 1 + x + \frac{x^2}{2} e^{x^+}$ (where $x^+$ equals $x$ if $x>0$ and $0$ otherwise), and the fact that  for any $y \in \eta $
		\[ (- \bar \ell^{y}_{B_R}(z) )^+ 
		 \leq 
		 \mathbb{E} \big[  \ell^y_{ B_{R}  }(z)  \big ]  , 
		 \]
 we gather that
 	\[  \begin{split} 
 	\mathbb{E}   \big[ e^{  - \lambda \bar \ell^{y}_{B_R}(z)  } \big ] 
 	& \leq 
 	 1 + \frac{ \lambda^2}{2} \mathbb{E} \big[ \big(  \bar \ell^{y}_{B_R}(z)\big)^2 \cdot e^{ \lambda ( -\bar \ell^{y}_{B_R}(z))^+ } \big ]
 	 \\ &  \leq 
	1 +  \frac{ \lambda^2}{2} \mathbb{E} \big[ \big( \bar \ell^{y}_{B_R}(z)\big)^2\big]  \cdot e^{\lambda \mathbb{E} [  \ell^y_{ B_{R}  }(z)  ]  } 
	\\ & \leq  
	\exp \Big( \frac{\lambda^2}{2}\mathbb{E} \big[ \big( \bar \ell^{y}_{B_R}(z)\big)^2\big] \cdot  e^{\lambda \mathbb{E} [  \ell^y_{ B_{R}  }(z)  ]  } \Big)  
	\end{split} 
 	\]
 for any $y \in \eta$. Putting all together, we obtain the following estimate for $z \in B_R \cup \partial B_R$:
 	\begin{equation}\label{startingpoint}
 	 \begin{split} 
 	 \mathbb{P}(\cS^\eta_{B_R}(z)=0)  
	\leq 
	\exp \bigg( 
	- \lambda \mathbb{E} & \big[\ell^{\eta}_{B_R}(z)-\ell^{B_R}_{B_R}(z)\big] +
	\\ & + \frac{\lambda^2}{2} \sum_{y \in \eta } 
	\mathbb{E} \big[ \big( \bar \ell^{y}_{B_R}(z)\big)^2\big] \cdot  e^{\lambda \mathbb{E} [  \ell^y_{ B_{R}  }(z)  ]  } \bigg) . 
	\end{split} 
 	\end{equation}
This will be our starting point for proving weak and strong inner bounds in the next sections. We remark that, while in classical IDLA the above probability can be bounded using only first moment bounds and large deviations for sums of independent Bernoulli, here we do need to control the second moment of the local time due to the fact that the number of particles is not conserved. This is a new feature introduced by the branching.

\section{No shape theorem in dimension 2}\label{sec:d2}
In this section we prove that BIDLA admits no shape theorem in dimension $d=2$ by showing that the process finds itself in an asymmetric state infinitely often. To formalise this, 
recall that  
for any $\varepsilon  >0$  the set $\mathcal{S}_\e$ denotes the set of $\varepsilon$-symmetric states, and if $A \notin  \mathcal{S}_\e $ the cluster $A$  is said to be  $\e$-asymmetric.

We prove  the $d=2$ statement in Theorem \ref{th:main}, namely that there exists $\varepsilon >0$ such that 
	\begin{equation}\label{eq:goal1} 
	\P ( A(t) \notin \mathcal{S}_{\varepsilon} \mbox{ for infinitely many } t ) =1 . 
	\end{equation}
To this end, it is convenient to introduce the following  time-change in the BIDLA process. 
Recall that  $A(0) = \emptyset$. Define inductively 
	\[\begin{cases}
	 \tau_1=1 , \\
	\tau_k = \inf \{ t\geq \tau_{k-1} : A(t) \neq A(\tau_{k-1} ) \} ,  
	\end{cases} 
	\] 
and let $\tilde{A} (0 ) := \emptyset$ and 
	\[ 
	\tilde{A} (k) := A(\tau_k ) , \qquad k\geq 1. \] 
Then the process $(\tilde{A} (t))_{t \geq 0}$ coincides with the original BIDLA process only taken at times when the cluster changes, i.e.\ it is the jump process associated with the Markov chain $(A(t))_{t \geq 0}$. 

To prove \eqref{eq:goal1} it clearly suffices to show that there exists $\varepsilon >0$ such that 
	\begin{equation} \label{eq:goal2}
	\P ( \tilde{A}(t) \notin \mathcal{S}_{\varepsilon} \mbox{ for infinitely many } t ) =1 . 
	\end{equation} 
To this end, we show that each time the process finds itself in an $\varepsilon$-symmetric state, it will jumps to an $\varepsilon$-asymmetric state with uniformly positive probability, as stated below. 

\begin{Proposition}\label{prop:breakdim2}
For any $\varepsilon  >0$ small enough, 
 there exists a positive constant $ c = c(\varepsilon )$, depending only on $\e$,  such that 
	\[ \P ( \tilde{A} (t) \notin \mathcal{S}_{\varepsilon}  \, 
	| \tilde{A} (t-1) \in \mathcal{S}_\varepsilon ) \geq c . \]
\end{Proposition} 
Note that, crucially, $c$ does not depend on $t$ nor on the  cluster $\tilde{A}(t)$. 

Assuming this proposition, the claim  \eqref{eq:goal2} follows immediately. Indeed, for any given $\varepsilon >0$ either the process visits $\mathcal{S}_\varepsilon$ finitely many times, in which case $\tilde{A}(t) \notin \mathcal{S}_\varepsilon$ eventually in $t$, or  the process visits $\mathcal{S}_\varepsilon$ infinitely often. On the latter event, by Proposition \ref{prop:breakdim2} the number of visits of the process to $\mathcal{S}_{\varepsilon}^c$ stochastically dominates an infinite sum of i.i.d.\ Bernoulli$(c)$ random variables, which is almost surely infinite. 

It thus remains to prove Proposition \ref{prop:breakdim2}.

\subsection{Proof of Proposition \ref{prop:breakdim2}} \label{sec:proofd2}
Let $A  \in \mathcal{S}_\varepsilon $ be any $\varepsilon$-symmetric cluster, and write  $\pi n^2 := |A|$ for the volume of $A$, so that $A $ is $\varepsilon $-close to the ball $B_n$ (note the slight abuse of notation: $n$ may not be an integer).

One step of (time-changed) BIDLA consists of releasing a BRW from the origin conditioned on reaching $\partial A$, and stabilising. We perform the stabilization procedure as follows: we first launch the BRWs from the origin, and show that most of the particles get close to the inner boundary of $A$ in a localised manner (cf.\ Lemmas \ref{le1dim2} and \ref{le2dim2}). We then release those particles, and show that with positive probability they create the desired asymmetry (cf.\ Lemma \ref{le3dim2}). Lastly, we release the remaining frozen particles and ask that they die  within distance $\e n$, so that they do not contribute significantly to the cluster's growth (cf.\ Lemma \ref{le4dim2}). See Figure \ref{fig:d2} for a depiction of this procedure.  \\

Fix an arbitrary $\e >0$, which will later be taken to be small enough. Introduce the following four events, depending on some parameters $\alpha , \beta , \gamma >0$ to be fixed later:
\begin{itemize}
\item $E_1= E_1 (\alpha , \beta )$ is the event that when releasing a BRW from the origin, conditioned on reaching $\partial B_{n(1-\e )}$ and stopped upon reaching $\partial B_{n(1-2\e )}$, the number of particles frozen on $\partial B_{n(1-2\e )}$ lies in the interval $[\alpha n^2, \beta n^2 ]$ (see Figure \ref{fig:d2} (b)). 
\item $E_2 = E_2 (\alpha , \beta ) $ is the event that, 
when releasing the first $\alpha n^2 $ particles frozen on $\partial B_{n(1-2\e )}$, each with freezing upon reaching distance $ \e n$ from its starting location, 
the following happens: 
\begin{itemize}
\item[(i)] exactly one of the pioneers survives to distance $\e n$, call its starting location $x$ (hence $x \in \partial B_{n(1-2\e )}$), 
\item[(ii)] the number of particles frozen on $\partial B_{\e n}(x)$ lies in the interval $[\alpha n^2, \beta n^2 ]$. 
\end{itemize}
These frozen particles, which are localised close to the inner boundary of $A$, are the ones that will create the desired asymmetry with positive probability (see Figure~\ref{fig:d2} (c)-(d)). 
\item $E_3 = E_3 ( \alpha , \gamma ) $ is the event that, when releasing the first $\alpha n^2 $ particles frozen on $\partial B_{\e n }(x)$ with freezing on $\partial B_{5\e n } (x)$, a given point $z \in \partial B_{5\e n } (x)$ is added to the cluster, and at most $\gamma n^2$ particles are frozen on $\partial B_{5\e n } (x)$ (see Figure \ref{fig:d2} (e)). 
\item $E_4 = E_4 ( \gamma , \delta  ) $ is the event that, when releasing the (at most  $\gamma n^2$) particles frozen on $\partial B_{5\e n } (x)$ and the (at most $(\beta - \alpha ) n^2$) pioneers frozen on $\partial B_{n(1-2\e)}$, no BRW travels by more than $\e n$ before extinction (see Figure \ref{fig:d2} (f)). 
\end{itemize}
\begin{figure}[!!ht]
    \begin{center}
     \centering
     \mbox{ 
  \includegraphics[width=.45\textwidth]{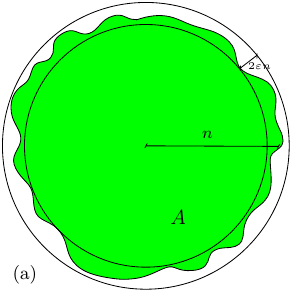}
   \qquad \includegraphics[width=.45\textwidth]{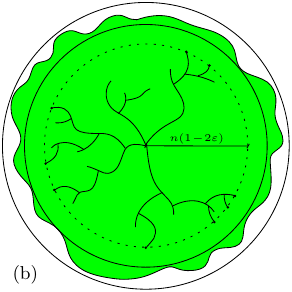}
   }
   \\ 
    \mbox{ \includegraphics[width=.45\textwidth]{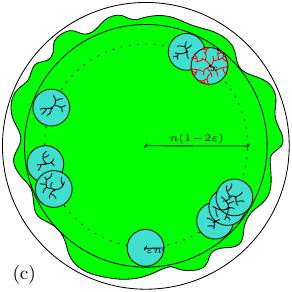}
  \qquad 
  \includegraphics[width=.45\textwidth]{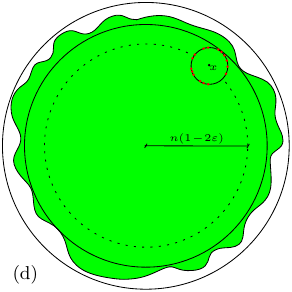}}
   \\ 
   \mbox{
    \includegraphics[width=.45\textwidth]{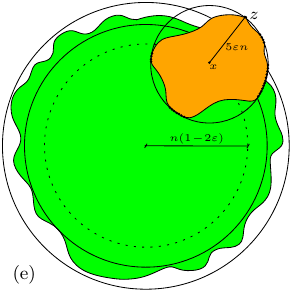}
   \qquad \includegraphics[width=.45\textwidth]{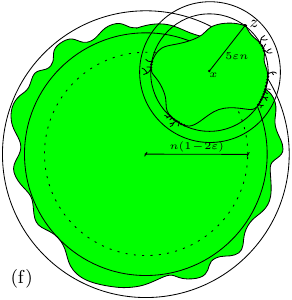}}
       \end{center}
 \caption{A depiction of the events $E_1 , E_2 , E_3 , E_4 $ defined above.} \label{fig:d2}
 \end{figure}

Note that on the event $E_1 \cap E_2 \cap E_3 \cap E_4 $ we have that the resulting BIDLA cluster after one step, call it $A' $, contains the point $z$, and so its radius exceeds $  n(1+2\e )$, while its volume has increased by at most  $\pi ( 5\e n)^2$, 
which implies that 
	\[  \sqrt{|A'|/\pi } \leq \sqrt{ ( 1 + 25\e^2)} \, n  . \]
It follows that, by taking $\e $ small enough so that 
	\begin{equation}\label{eps_small}
	 1+2\e 
	>\sqrt{ ( 1 + 25\e^2)} , 
	\end{equation}
the cluster $A'$ not $\e$-symmetric. 

The above reasoning shows that  it suffices to prove that 
	\[ \P ( E_1 \cap E_2 \cap E_3  \cap E_4) \geq c(\varepsilon ) >0. \]
We do so in  Lemmas \ref{le1dim2}, \ref{le2dim2}, \ref{le3dim2} and \ref{le4dim2} below. 
	
\begin{Lemma}\label{le1dim2}
For any $\alpha >0$ and for any $\beta >0 $ large enough, there exists a positive constant $c_1 = c_1 ( \alpha , \beta ) $ such that $\P (E_1  ) \geq c_1$. 
\end{Lemma}
\begin{proof}
Recall that for a given subset $K\subset \mathbb{Z}^2$  we denote by $\ell_K^0 (z) $ the number of particles frozen at $z \in \partial K $ when we launch a BRW from the origin with freezing upon exiting $K$. 
Then 
	\[ \P (E_1) = \P \Big( \ell_{B_{n(1-\e )}}^0 ( \partial B_{n(1-2\e )} )  \in [\alpha n^2 , \beta n^2 ]
	\, \Big| \ell_{B_{n(1-\e )}}^0 ( \partial B_{n(1-\e )} ) >0 \Big) 
	\] 
which exceeds $c_1$ by Corollary \ref{cor:conditioned}. 
\end{proof}

\begin{Lemma}\label{le2dim2}
For any $\alpha >0$ and for any $\beta >0 $ large enough, there exists a positive constant $c_2 = c_2 ( \alpha , \beta ) $ such that $\P (E_2  | E_1 ) \geq c_2$. 
\end{Lemma}
\begin{proof}
We start by noticing that  each one of the $\alpha n^2$ BRWs moves independently with no settling up to extinction or reaching distance $\e n $ from its starting point. It follows that if $\{ \ell_{B_{\e n }}^{(k)} : 1\leq k \leq \alpha n^2 \} $ denote i.i.d.\ copies of $\ell_{B_{\e n }}^0 ( \partial B_{\e n } ) $, we have to show that 
	\[ \begin{split} 
	\P (E_2 | E_1 ) & = \P \bigg( \bigcup_{k=1}^{\alpha n^2} \{ \ell^{(k)}_{B_{\e n }} \in [\alpha n^2 , \beta n^2 ] \} \cap \bigg\{ \bigcap_{j=1,\, j\neq k}^{\alpha n^2} \{  \ell^{(j)}_{B_{\e n }} =0 \} \bigg\} \bigg) 
	\\ & = \alpha n^2 \cdot  \P \Big( \ell_{B_{\e n }}^0 (\partial B_{\e n }) \in [\alpha n^2 , \beta n^2 ] \Big) 
	\cdot \P \Big(  \ell_{B_{\e n }}^0(\partial B_{\e n } )  =0\Big) ^{\alpha n^2 -1} 
	\end{split} 
	\] 
exceeds $c_2$. Indeed, 
by Proposition \ref{le:distance} we find 
	\[ \P \Big(  \ell_{B_{\e n }}^0(\partial B_{\e n } )  =0\Big) ^{\alpha n^2 -1} 
	\geq \Big( 1 - \frac{C}{(\e n )^2} \Big)^{\alpha n^2 -1} 
	\geq c ( \e , \alpha ) \]
for some constant $c(\e , \alpha ) >0$. Moreover, 
	\[ \begin{split} 
	 \P \Big( \ell_{B_{\e n }}^0 (\partial B_{\e n }) \in [\alpha n^2 , \beta n^2 ] \Big) 
	& = \underbrace{\P \Big( \ell_{B_{\e n }}^0 (\partial B_{\e n }) \in [\alpha n^2 , \beta n^2 ] \Big| 
	\ell_{B_{\e n }}^0 (\partial B_{\e n }) > 0 \Big)}_{\geq 1 - c (\alpha , \beta )/n^2 } \cdot 
	\underbrace{ \P \Big( \ell_{B_{\e n }}^0 (\partial B_{\e n } ) >0 \Big) }_{\geq c/(\e n )^2 }
	\\ & \geq \alpha n^2 \cdot \Big(1 - \frac{c (\alpha , \beta )}{n^2} \Big) \cdot \frac{c}{\e^2 n^2 } \geq 
	 \frac{c\alpha}{2\e^2} 
	\end{split} 
	\]
for any $\alpha >0$ and for $\beta $ large enough (larger than an absolute constant) by Corollary \ref{cor:conditioned} and  Proposition \ref{le:distance}. This concludes the proof. 
\end{proof}

\begin{Lemma}\label{le3dim2}
For any $\alpha  >0$ there exists $\gamma = \gamma  (\alpha)$ such that  $\P (E_3 | E_1 \cap E_2  ) \geq c_3$ for some positive constant $c_3 = c_3 ( \alpha , \gamma )$. 
\end{Lemma}
\begin{proof} 
Recall from Section \ref{sec:domain} 
that, if $\eta$ is a particle configuration, for any given $z \in \partial B_{5\e n }(0)$ we denote by $\mathscr{S}^\eta_{B_{5\e n }} (z) $ the number of particles frozen at $z$ in the BIDLA stabilization of $\eta$ with freezing on $\partial B_{5\e n }$. Thus, if $\eta$ denotes an arbitrary configuration of $\alpha n^2$ particles supported on $\partial B_{\e n }(0)$, we have that 
	\[ \P (E_3 | E_1 \cap E_2 ) \geq 
	\P \Big( \Big\{ \mathscr{S}^\eta_{B_{5\e n }} (z) \geq 1 \Big\} \cap 
	\Big\{ \mathscr{S}^\eta_{B_{5\e n }} \big( \partial B_{5\e n } \big)  \leq \gamma n^2 \Big\} \Big ) . 
	 \]
It therefore suffices to show that the latter probability is larger than $c_3 >0$. To this end, we prove that 
	\begin{equation}\label{sum_dim2}
	 \P \Big(  \mathscr{S}^\eta_{B_{5\e n }} (z) =0 \Big) + 
	\P \Big(  \mathscr{S}^\eta_{B_{5\e n }} \big( \partial B_{5\e n } \big) > \gamma n^2 \Big )
	<1 
	\end{equation} 
by showing that the first probability in the right hand side is $<1$, while the second one can be made arbitrarily small by choosing $\gamma$ large enough, depending on $\alpha$. 

Recall the definition of $\sigma^2 $ from \eqref{assumption}. As a first step, we claim that 
	\begin{equation}\label{step1dim2}
	 \P \big(  \mathscr{S}^\eta_{B_{5\e n }} (z) =0 \big)  
	 \leq e^{ -C' \alpha / \varepsilon^2 } 
	 \end{equation}
for some explicit constant $C' = C' (\sigma^2)$ depending only on $\sigma^2$. 
Indeed, reasoning as in Section \ref{sec:classical}, by \eqref{startingpoint} with $R = 5\e n $ and any $\lambda >0$ we have that 
	\[  	 \mathbb{P}(\cS^\eta_{B_R}(z)=0)  
	\leq 
	\exp \bigg( 
	- \lambda \mathbb{E}  \big[\ell^{\eta}_{B_R}(z)-\ell^{B_R}_{B_R}(z)\big] 
	 + \frac{\lambda^2}{2} \sum_{y \in \eta } 
	\mathbb{E} \big[ \big( \bar \ell^{y}_{B_R}(z)\big)^2\big] \cdot  e^{\lambda \mathbb{E} [  \ell^y_{ B_{R}  }(z)  ]  } \bigg) . 
	\]
Now, for any $y \in \eta \subseteq B_{R/2}$ we can bound $ 
		 \mathbb{E} \big[  \ell^y_{ B_{R}  }(z)  \big ] 
		 = G_{R} (y,z) \leq C/R $ 
for some absolute constant $C<\infty$ (cf.\ \cite{asselah2013logarithmic}, Lemma 3.4). Moreover, 
by Lemma \ref{lem-2m-recurrent}, 
	\[ \sup_{y\in \eta } 
	\mathbb{E} \big[ \big( \bar \ell^{y}_{B_R}(z)\big)^2\big]
	\leq 
	\sigma^2 \sup_{y\in \eta } \sum_{ w\in B_{R}} G_{R}(y,w) ( G_{R}(w,z))^2
	\leq C' 
	\] 
with $C' = C'(\sigma ) = 5  C_2 \sigma^2$ with $C_2$ as in  \eqref{green-4}. 
We thus  end up with 
	\[  \prod_{y \in \eta } 
 	\mathbb{E}   \big[ e^{  - \lambda \bar \ell^{y}_{B_R}(z)  } \big ] 
  	\leq
	\exp\Big( \alpha n^2  \frac{C'}{2}   \lambda^2 e^{\lambda C/R} \Big) .  
	\]
Plugging this into \eqref{eq:Nzdim2}, together with 
	\[ \mathbb{E}\big[\ell^{\eta}_{B_R}(z) \big] 
	\geq 
	 \alpha n^2 \inf_{y \in \eta } G_{R}(y,z) 
	 \geq \alpha n^2 \frac{c}{R} = \frac{\alpha c n}{5\e} \]
with $c$ absolute constant, and 
	\[ \E \big[ \ell^{B_R}_{B_R}(z)\big] 
	 = \sum_{y \in B_{R} } G_{R}(y,z) 
	\leq c' R \, 
	\leq \frac 12 \mathbb{E}\big[\ell^{\eta}_{B_R}(z) \big] 
	\]
for $\e$ small enough, 
we conclude that 
	\[ \mathbb{P}(\cS^\eta_{B_R}(z)=0) 
	\leq 
	\exp\Big( \alpha n^2  \frac{C'}{2}   \lambda^2 e^{\lambda C/R} 
	- \lambda \frac{\alpha cn}{5 \e } 
	\Big)
	= \exp\Big\{ \lambda \alpha n 
	\Big( \frac {C'} 2 n \lambda  e^{\lambda C/R} - \frac c {5\e} \Big) 
	\Big\} 
	. \]
Take $\lambda = \delta /n $ with $\delta  = c / (10 \e C' ) $  to get that 
	\[ \mathbb{P}(\cS^\eta_{B_R}(z)=0) 
	\leq e^{- \alpha \delta c / (10 \e ) } 
	\leq e^{ -  C'' \alpha /\e^2   }   \]
for $C''$ absolute constant. 
This proves \eqref{step1dim2}.

We now show that 
	\begin{equation}\label{step2dim2}
	\P \Big(  \mathscr{S}^\eta_{B_{R}} \big( \partial B_{R} \big) > \gamma n^2 \Big )
	\leq \frac{\alpha}{2\gamma} . 
	\end{equation} 
To this end, note that 
	\[ \E \big[  \mathscr{S}^\eta_{B_{R}} \big( \partial B_{R} ) \big]
	\leq \alpha n^2 \sup_{y \in \eta  } \E \big[ \ell^y_{B_{R}} ( \partial B_{R}) \big] 
	= \alpha n^2    \]
since $\E \big[ \ell^y_{B_{R}} ( \partial B_{R}) \big] =1$ for all $y \in \eta $. 
It follows that 
	\[ \P \big(  \mathscr{S}^\eta_{B_{R}} \big( \partial B_{R} \big) > \gamma n^2  \big)	
	\leq 
	\frac{ \E \big[  \mathscr{S}^\eta_{B_{R}} \big( \partial B_{R} ) \big] }{\gamma n^2} 
	\leq   \frac{\alpha }{\gamma } .  \]
Plugging \eqref{step1dim2} and \eqref{step2dim2} into \eqref{sum_dim2}, and recalling that $R = 5\e n $,  we end up with  
	\[ \P \Big(  \mathscr{S}^\eta_{B_{5\e n }} (z) =0 \Big) + 
	\P \Big(  \mathscr{S}^\eta_{B_{5\e n }} \big( \partial B_{5\e n } \big) > \gamma n^2 \Big )
	\leq 
	 e^{- C' \alpha / \varepsilon^2  }   + \frac{\alpha }{\gamma } , \]
which can be made  strictly smaller than $1$ for any $\alpha >0$ by taking  $\gamma $ large enough. 
\end{proof}

\begin{Lemma}\label{le4dim2}
For any $\beta, \gamma , \e  >0$ there exists a positive constant $c_4 = c_4 ( \beta , \gamma , \e )$ such that   $\P (E_4 | E_1 \cap E_2 \cap E_3 ) \geq c_4$. 
\end{Lemma}
\begin{proof}
Note that, since there are at most $(\beta - \alpha + \gamma ) n^2 $ particles that move independently up to distance $\e n $ from their starting locations, 
	\[ \P ( E_4 | E_1 \cap E_2 \cap E_3 ) \geq
	\P \Big( \ell^0_{B_{5\e n } } ( \partial B_{5\e n } ) =0 \Big)^{ (\beta  + \gamma ) n^2 }, \]
so it suffices to show that the latter probability is larger than $c_4$. Indeed, by Proposition \ref{le:distance} we have that 
	\[ \P \Big( \ell^0_{B_{5\e n } } ( \partial B_{5\e n } ) =0 \Big) \geq 1- \frac{C}{(\e n )^2} 
	\] 
for some absolute constant $C>0$. Thus we find that 
	\[ \P \Big( \ell^0_{B_{5\e n } } ( \partial B_{5\e n } ) =0 \Big)^{ (\beta  + \gamma ) n^2 } \geq 
	\Big( 1- \frac{C}{(\e n) ^2} \Big) ^{ (\beta + \gamma ) n^2 }
	\geq  c_4 (\beta , \gamma , \e ) >0 , \]
as wanted. 
\end{proof}

We can now prove that, once in a symmetric state, the (time-changed) BIDLA process becomes asymmetric in one step with uniformly positive probability. 
\begin{proof}[Proof of Proposition \ref{prop:breakdim2}] 
Let $A  \in \mathcal{S}_\varepsilon $ be any $\varepsilon$-symmetric cluster, and write  $\pi n^2 := |A|$ for the volume of $A$. 
Choose $\alpha >0$ arbitrary and $\beta >0$ large enough so that 
	\[ \P (E_1 \cap E_2 ) \geq c_1 ( \alpha , \beta )\cdot c_2 (\alpha , \beta )  >0 \]
by Lemmas \ref{le1dim2} and \ref{le2dim2} above. Now choose $\gamma = \gamma (\alpha ) $ as in Lemma \ref{le3dim2} above so that 
	\[ \P (E_3 | E_1 \cap E_2 ) \geq c_3 ( \alpha , \gamma ) > 0 . \]
Finally, for these choices of $\alpha , \beta , \gamma$  by Lemma \ref{le4dim2} 
	\[ \P ( E_4 | E_1 \cap E_2 \cap E_3  ) \geq c_4 ( \beta , \gamma , \e ) >0 ,  \]
which concludes the proof. 
\end{proof}

\section{The inner bound for $d>2$}\label{sec:innerbound}
In this section we prove that, when releasing $|B_n |  $ BRWs with settling on $\mathbb{Z}^d$ with $d>2$, the resulting particle configuration covers the ball $B_R$ for some $R = n - n^\alpha$ with high probability. Here the constant $\alpha <1$ is to be chosen later as small as possible. 

To see this, we follow the classical approach to the  inner bound sketched in Section \ref{sec:classical}, together with necessary second moment bounds on the local time of critical BRW inside Euclidean balls.

\subsection{Proof of the inner bound}
\begin{Proposition}\label{prop:lowerbound}
Suppose that $d > 2$, and let $\eta := |B_n| \1_{0}$. There exists $ c_d>0$  such that,
for each  any $\alpha>1/2$ and $n \in \mathbb{N}$,
$$
\mathbb{P}\big( B_{n - n^\alpha } \not\subseteq  \mathscr{S} (\eta )
\big ) \leq  n^{d}\cdot \exp \Big(  - c _d\cdot n^{2\alpha-1}\Big(\frac{\1_{d=3}}{\log n}+\1_{d>3}\Big)\Big).
$$
\end{Proposition}
\begin{proof} 
It suffices to prove the statement for $n$ large enough. 
Note that  
	\[ \mathbb{P}\big( B_{n - n^\alpha } \not\subseteq  \mathscr{S} (\eta ) \big) 
	\leq \sum_{z \in B_{n-n^\alpha } } \mathbb{P}\big( \mathscr{S}^\eta (z) = 0 \big) 
	= \sum_{R=1}^{n-n^\alpha } \sum_{z \in \partial B_R} \P \big( \mathscr{S}^\eta (z) = 0 \big) 
	\]
and $\P ( \mathscr{S}^\eta (z) = 0  ) \leq \P ( \mathscr{S}_{B_R}^\eta (z) = 0 )$. We now show that   
\begin{equation}\label{key-covering}
\sup_{R\le n-n^\alpha} \sup_{z\in \partial B_R} \mathbb{P}(\cS^{\eta }_{B_R}(z)=0) \le
\exp \Big(  - c _d\cdot n^{2\alpha-1}\Big(\frac{\1_{d=3}}{\log n}+\1_{d>3}\Big)\Big).
\end{equation}
To see this, take $R\leq n-n^\alpha$ arbitrary. Then by \eqref{startingpoint} 
with $\eta = |B_n| \1_0 $, using that $|B_n| \leq 2 \omega_d  n^d$,  
we have 
	\begin{equation}\label{newstarting}
	 \mathbb{P}(\cS^\eta_{B_R}(z)=0)  
	\leq 
	\exp \bigg( 
	- \lambda \mathbb{E}  \big[\ell^{\eta}_{B_R}(z)-\ell^{B_R}_{B_R}(z)\big] +
	 \lambda^2 \omega_d n^d \cdot 
	\mathbb{E} \big[ \big( \bar \ell^{0}_{B_R}(z)\big)^2\big] \cdot  e^{\lambda \mathbb{E} [  \ell^0_{ B_{R}  }(z)  ]  } \bigg)  
	\end{equation} 
for any $z \in \partial B_R$.

For the first term we  find 
	\begin{equation}\label{firstpart}
	 \begin{split} 
	\mathbb{E}  \big[\ell^{\eta}_{B_R}(z)-\ell^{B_R}_{B_R}(z)\big] 
	& = |B_n|  G_R(0,z) - \sum_{y\in B_R} G_R(y,z) 
	\\ &  = ( |B_n| - |B_R| ) G_R (0,z) + 
	 \Big( |B_R| \cdot G_R(0,z) - \sum_{y\in B_R} G_R(y,z) \Big) 
	\\ & \geq  ( n^d - (n-n^\alpha )^d ) \cdot \frac{c}{ R^{d-1}}   - C
	\geq \kappa_d n^\alpha  , 
	\end{split} 
	\end{equation}
where the first inequality follows from Lemma~\ref{lem-Harm} and the fact that $G_R(0,z)\ge c   R^{1-d}$, while the last one holds for some constant $\kappa_d$ depending only on the dimension. 

For the second term we have that 
$ \mathbb{E} [  \ell^0_{ B_{R}  }(z)  ] = G_R(0,z) \leq C / R^{d-1} $ while, 
by Lemmas \ref{lem-l2} and \ref{lem-2moment}, 
	\begin{equation}\label{secondpart}
	\begin{split} 
	\mathbb{E} \big[ \big( \bar \ell^{0}_{B_R}(z)\big)^2\big] 
	& \leq 
	G_R(0,z) + \sigma^2 \sum_{y \in B_R } G_R(0,y) G_R^2 (y,z) 
	\\ & \leq 
	C_{d , \sigma^2} \cdot \frac{1}{R^{d-1}} \cdot \Big( \1_{d=3} \log R + \1_{d>3} \Big) 
	\end{split} 
	\end{equation}
for some constant $C_{d,\sigma^2}$ depending only on the dimension $d$ and the offspring variance $\sigma^2$. 
Plugging \eqref{firstpart} and \eqref{secondpart} into \eqref{newstarting} we end up with 
	\[ \sup_{R \leq n-n^\alpha } \mathbb{P}(\cS^\eta_{B_R}(z)=0)  
	\leq 
	\exp \bigg( - \lambda \kappa_d n^\alpha 
	+ \lambda^2  C_{d , \sigma^2} \cdot n  \cdot \Big( \1_{d=3} \log n + \1_{d>3} \Big) \bigg) .  \]
Optimising over $\lambda$ yields the desired result. 
\end{proof}  

Along the same lines one can prove the inner bound stated in Lemma \ref{lem-covering}, dealing with a more general initial configuration. As the lemma will not be used in the remainder of the paper, we sketch the argument and leave the details to the reader. 

\begin{proof}[Proof of Lemma~\ref{lem-covering}]
The lower bound relies on the scenario that no branching random walk starting inside $B_{n/2}$ reaches  $\partial B_n$ before extinction, which has probability at least  
\[
\mathbb{P}(\text{out of $|\eta | $  branching random walks none reaches }\partial B_n)\ge
\Big( 1-\frac{C}{n^2} \Big)^{\alpha\cdot 2 \omega_d n^d}
\]
by Proposition \ref{le:distance}, since $|\eta | \leq \alpha \cdot |B_n| = \alpha \cdot 2\omega_d  n^d $. 

For the upper bound take $\eta$ supported in $B_{n/2}$ with $|\eta | \geq \alpha |B_n| $. Reasoning as in  the proof of Proposition \ref{prop:lowerbound} above, one shows  that 
\begin{equation}\label{eq:cover-1}
\sup_{2n/3\le R\le n} \sup_{z\in \partial B_R} \mathbb{P}(\cS^\eta_{B_R}(z)=0) \le
 \exp\Big(-C_d \frac{\alpha\cdot n}{\log n\cdot \1_{d=3}+\1_{d>3}}\Big).
\end{equation}
This proves that when releasing BIDLA particles from $\eta$ with freezing on $\partial B_n$, the annulus $B_n \setminus B_{2n/3}$ is covered with overwhelming probability. It remains to exclude holes inside $B_{2n/3}$. 
To this end, we could iterate this procedure $k$ times, for some large integer $k \geq \alpha$, 
so as to cover each site of $B_n\setminus B_{2n/3}$ with $k$ individuals. 
By the same argument, stabilising this configuration will result in covering the ball $B_{2n/3}$, with overwhelming probability. 
\end{proof}

\subsection{Number of particles frozen on $\partial B_n$}
\label{sect:controllingthenumber}
We proved in the previous section that releasing $ |B_n| $ particles from the origin according to BIDLA dynamics one has that  the event 
	\[\cE_n(\alpha):= \{B_{n-n^\alpha }\subseteq \mathscr{S}^\eta_{B_n}  \}\] 
holds with high probability as soon as $\alpha>1/2$. This section focuses on bounding the number of particles, or \emph{pioneers}, frozen on $\partial B_n$. More precisely, we 
 show that on the good event $\cE_n(\alpha ) $ there are at most 
 $O(n^{d-1+\alpha})$ frozen particles on  $ \partial B_n$. Note that, while  this type of bound would be trivial for classical IDLA, the creation of particles in BIDLA means that, in principle, the number of pioneers could be arbitrarily large. We show that this number  can be bounded by re-adapting the ghost particle approach in combination with the BIDLA inner bound proved in the previous section.

For an arbitrary particle configuration $\eta$ write 
	\[ \cS^{\eta}_{B_n}(\partial B_n)=\sum_{z\in \partial B_n}\cS^{\eta}_{B_n}(z) , 
	\qquad 
	\ell^\eta_{B_n} (\partial B_n)=\sum_{z\in \partial B_n}\ell^{\eta}_{B_n}(z) . 
	\]
Thus $\cS^{\eta}_{B_n}(\partial B_n)$ (respectively $\ell^\eta_{B_n} (\partial B_n)$) counts  the number of pioneers frozen on $\partial B_n$ when releasing BIDLA particles (respectively independent BRWs) from each location of $\eta$ 
 with  freezing on $\partial B_n$. 

Now take $\eta = |B_n| \1_0$. Setting  $\eta_G:=\cS^{\eta }_{B_n}\cap B_n$, as in \eqref{eq:centralequality} we  have 
\begin{equation}\label{eq:local-equality}
\cS^{\eta}_{B_n}(\partial B_n) +
\tilde \ell_{B_n}^{\eta_G}(\partial B_n) =\ell_{B_n}^{\eta }(\partial B_n) .
\end{equation}
Moreover,  on the good event $\cE_n(\alpha)$, 
\begin{equation}\label{eq:local-ineq}
\cS^{\eta }_{B_n}(\partial B_n) +
\tilde \ell_{B_n}^{B_R}(\partial B_n) \le
\ell_{B_n}^{\eta }(\partial B_n) , 
\end{equation}
where we have set   $R=n-n^\alpha$ for brevity. Notice that compared to \eqref{eq:centralinequality} now the inequality is reversed, which is a consequence of the event $\cE_n (\alpha )$ holding true. This provides an upper bound for the number of pioneers $\cS^{\eta }_{B_n}(\partial B_n) $ in terms of independent random variables, which can be used to deduce the following result. 
\begin{Proposition}\label{prop-standby}
Assume $d>2$ and take $\eta = |B_n| \1_0$. 
For any $\alpha>1/2$, and $\beta$ large enough, we have for some $\kappa>0$
and any $n$ large enough,
\begin{equation}\label{ineq-standby}
\mathbb{P}( \ \cS^{\eta }_{B_n}(\partial B_n)> \beta\cdot n^{d-1+\alpha})\le
n^{d}\cdot \exp \Big(  - c _d\cdot n^{2\alpha-1}\Big(\frac{\1_{d=3}}{\log n}+\1_{d>3}\Big)\Big).
\end{equation}
\end{Proposition}
\begin{proof} With $\cE_n ( \alpha )$ defined as above, clearly 
\[
\mathbb{P}(\cS^{\eta }_{B_n}(\partial B_n)>\beta\cdot n^{d-1+\alpha}) \le
\mathbb{P}(\cE_n(\alpha), \cS^{\eta }_{B_n}(\partial B_n)>\beta\cdot n^{d-1+\alpha})
+\mathbb{P}(\cE_n(\alpha)^c) .
\]
The term $\mathbb{P}(\cE_n(\alpha)^c)$ is dealt with in Proposition~\ref{prop:lowerbound}, and will be the dominant term. For the other term use that 
$\cE_n(\alpha)$ and $\cS^{\eta}_{B_n}(\partial B_n)$ 
are independent from $\tilde \ell_{B_n}^{B_R}(\partial B_n)$ to obtain 
	\begin{equation}\label{eq:step10}
	\begin{split}
	\P (\cE_n(\alpha),\ \cS^{\eta }_{B_n}(\partial B_n)>\beta n^{d-1+\alpha })  
	& \leq 
	e^{-\lambda \beta\cdot n^{d-1+\alpha }}
	\frac{ \mathbb{E} [ \1_{\cE_n(\alpha)} 
	\cdot \exp( \lambda \cS^{\eta }_{B_n}(\partial B_n)) ]
	\cdot \mathbb{E}  [e^{ \lambda \ell_{B_n}^{B_R}(\partial B_n) } ]} 
	{ \mathbb{E} [\exp\big(\lambda  \ell_{B_n}^{B_R}(\partial B_n)\big)  ]} 
	 \\ &
	 \leq e^{-\lambda\cdot \beta\cdot n^{d-1+\alpha}}
	 \frac{ \mathbb{E} [\exp\big( \lambda  \ell_{B_n}^{\eta}(\partial B_n) \big) ]  }
	 {\mathbb{E}  \big[ \exp\big(\lambda \ell_{B_n}^{B_R}(\partial B_n)\big) \big]   } 
	 \\ & 
	 =  e^{-\lambda\cdot \beta\cdot n^{d-1+\alpha}}
 	\frac{ \mathbb{E} \big[ e^{ \lambda  \bar \ell_{B_n}^{\eta}(\partial B_n)  }  \big] }
 	{\mathbb{E}  \big[ e^{ \lambda  \bar  \ell_{B_n}^{B_R}(\partial B_n) }\big]  }
	e^{\lambda  \mathbb{E}[ \ell_{B_n}^{\eta}(\partial B_n)- \ell_{B_n}^{B_R}(\partial 		B_n)]}
	\\ & \leq  
	e^{-\lambda\cdot \beta\cdot n^{d-1+\alpha}} \Big( 
 \mathbb{E}   \Big[ e^{  \lambda\cdot \bar  \ell_{B_n}^{0}(\partial B_n)  } \Big ] 
 \Big)^{|B_n| } \, \, 
e^{\lambda\mathbb{E}[ \ell_{B_n}^{\eta}(\partial B_n)- \ell_{B_n}^{B_R}(\partial B_n)]}.
	 \end{split}
	\end{equation}
Here again  $\bar  \ell_{B_n}^{0}(\partial B_n) := \ell_{B_n}^{0}(\partial B_n)  - \E [ \ell_{B_n}^{0}(\partial B_n)  ] $. 
Write $\ell_n := \ell_{B_n}^{0}(\partial B_n) $ and $\bar \ell_n := \bar \ell_{B_n}^{0}(\partial B_n)  $ to lighten the notation. 

By 
\eqref{eq:Amine_centered} in Proposition \ref{prop:largedeviation} 
we have that,  for $\gamma $ small enough,  
	\begin{equation}\label{eq:step11}
	\mathbb{E} \Big[ \exp \Big( \gamma \bar \ell_n / n^2  \Big) \Big] 
	\le \exp \Big( \kappa_d \gamma^2 / n^2 \Big) 
	\end{equation}
for some constant $\kappa_d$ depending only on $d$. 
Moreover, since 
$ G_n(0, \partial B_n)=G_n(y, \partial B_n) = 1$ for any $y \in B_n$, 
	\[ \begin{split} 
	\mathbb{E}[ \ell_{B_n}^{\eta}(\partial B_n)- \ell_{B_n}^{B_R}(\partial B_n)]
	 & =  |B_n| G_n(0,\partial B_n)-\sum_{y\in B_R}\sum_{z\in \partial B_n} G_n(y,z)
	 \\ & =|B_n|-|B_R|=O(n^{d-1 + \alpha }).
	 \end{split} 
	\]
It follows by choosing   $\lambda=\gamma\cdot n^{-2}$ that 
	\begin{equation}\label{eq:step12}
	\mathbb{P}(\cE_n(\alpha),\ \cS^\eta_{B_n}(\partial B_n)>\beta\cdot n^{d-1+\alpha}) 
	\le
	\exp\big(- \gamma \beta n^{d-3+\alpha}+ 2 \gamma^2 \omega_d \kappa_d n^{d-2} +\frac{\gamma}{n^2} C_d n^{d-1 + \alpha } \Big).
	\end{equation}
Thus, if we choose now $\beta\ge 2 C_d$, and $\gamma$ proportional to $n^{\alpha-1}$, we gather that 
\begin{equation}\label{eq:step13}
\mathbb{P}(\cE_n(\alpha),\ \cS^\eta_{B_n}(\partial B_n)>\beta\cdot n^{d-1+\alpha}) \le 
\exp\Big( -\frac{\beta}{2}\cdot\gamma\cdot n^{d-3+\alpha}+ 2\gamma^2 \omega_d \kappa_d n^{d-2} \Big)
\le e^{-\kappa\cdot \beta\cdot n^{d-3+2\alpha-1}}
\end{equation}
for some $\kappa>0$, which concludes the proof. 
\end{proof} 

\section{The outer bound for $d>2$} \label{sec:outerbound}
In this section we provide an outer bound for the aggregate starting from
an arbitrary particle configuration.
By combining these results with those of
Section \ref{sect:controllingthenumber},
this 
 will lead to the proof of the shape theorem in dimension $d \geq 3$.
The main result of this section is Theorem \ref{thm:controlling} below.
Recall that  $\mathscr{S}(\eta)$ denotes the aggregate of the BIDLA
process starting from a general initial particle configuration $\eta : \Z^d  \to \mathbb{N}$. 
   \begin{Theorem}\label{thm:controlling}
 Let $d \geq 3$.   There  exists $\alpha_0>0$, $c >0$ and $N_0 < \infty$ such that 
 the following holds. 
 For any  $\alpha > \alpha_0$, $N> N_0$, 
  $
  R \leq \frac{1}{\alpha}  \sqrt{  \frac{\log N} {N} },
   $
  any particle configuration $\eta$ with support 
   $\partial B_{ R }$  and number of particles $N$, we have 
   $$
   \mathbb{P}  \Big ( \mathscr{S}(\eta) \not\subset B_{  R +  c \,    \alpha^2 \,  {N }^{\frac{1}{d}} {\log (N)}^\rho }  \Big ) \leq 
    \frac{1}{ N^{\alpha}},
  $$
  where $\rho=\frac{7}{2}$ if $d=3$ and $\rho=1$ if $d>3$. 
   \end{Theorem}
 The proof is structured as follows. In Section \ref{sect:outerboundgrowth} we consider one BRW starting inside a set $K$, and lower bound the support of the pioneers on $\partial K$  in terms of the second moment of the local time (cf.\ Lemma \ref{lemma:connectionsquarelocalsettl}).  In Section \ref{sec:RBG} we define the Random Barrier Growth process (in short RBG),  and prove the main contraction result (cf.\ Proposition \ref{prop:settlmentsphericalprocess}). In Section \ref{sec:RBGshells} we upper bound the growth of the RBG process by considering a sequence of shells, large enough to extinct the process. We then conclude the proof of Theorem \ref{thm:controlling} in Section \ref{sec:controlling}. Finally, in Section \ref{sec:finalsection} we conclude the proof of Theorem \ref{th:main} by proving  Theorem \ref{th:deviation}, which implies it. 

 \subsection{Expected number of particles settling on a layer}
 \label{sect:outerboundgrowth}  
For any finite $K \subset \mathbb{Z}^d$ and $x \in K$ we define 
 $$
 \mathcal{R}_K^x := \sum\limits_{z \in \partial K}  
 \1_{ \{\ell^x_K(z) > 0\}},  \quad \quad \Gamma^x_K := \sum\limits_{z  \in \partial K } (\ell_K^x(z))^2.
 $$
Here $\mathcal{R}_K^x$ measures the support of the pioneers on $\partial K$. 
The next  lemma relates  the expectations of the above random variables.
\begin{Lemma}\label{lemma:connectionsquarelocalsettl}
For any $\delta \in (0, 1)$ and any $x \in K$ we have that 
$$
\mathbb{E}[ \mathcal{R}_K^x  ] \geq \delta \big  (1 - \delta \,  \mathbb{E}[\Gamma_K^x  ]  \big ).
$$
\end{Lemma}
\begin{proof}
Suppose that $x \in K$. 
Recall the definition of 
the number of pioneers  $\ell_K^x(\partial K)$ in 
(\ref{eq:pioneers}).
By  the Cauchy-Schwarz inequality we  have that 
$$
( \ell_K^x(\partial K)) ^2  \leq \Gamma_K^x  \,  \mathcal{R}^x_K.
$$
Thus, 
\begin{equation}\label{eq:inequalitygammarho}
\Gamma_K^x \geq \frac{ (\ell_K^x(\partial K) )^2}{\mathcal{R}^x_K} 
\1_{\{  \mathcal{R}^x \leq \delta   \, { \ell_K^x(\partial K)} \,   \} }
\geq \frac{1}{\delta} 
\1_{ \{   \mathcal{R}^x_K \leq \delta  { \ell_K^x(\partial K)} \}   }  { { \ell_K^x(\partial K)}}.
\end{equation}
This implies that 
\begin{align*}
\mathcal{R}^x_K 
& \geq  \delta \, {\ell_K^x(\partial K)}
\1_{\{  \mathcal{R}_K^x > \delta \, { \ell_K^x(\partial K)}  \} } \\
& \geq  \delta  \big ({\ell_K^x(\partial K)} \, -  \, { \ell_K^x(\partial K)} \1_{\{  \mathcal{R}_K^x \leq \delta \, { \ell_K^x(\partial K)}   \} }    \big )  \\
& \geq 
\delta 
\big ({ \ell_K^x(\partial K)} \,  -  \,  \delta  \,  \Gamma_K^x     \big ).
 \end{align*}
 Taking expectations on both sides and using that $\mathbb{E}[{ \ell_K^x(\partial K)} ] = 1$
 we conclude the proof. 
\end{proof}

The next proposition plays a central role in the proof of our 
outer bound.

\begin{Proposition}\label{prop:fractionsettle}
There exists  $c > 0 $  such that,
for each $R >0$ and $x \in B_R$,
we have
$$
\mathbb{E}
\bigg [
  \sum\limits_{z \in \partial B_R  } \big ( 
  \mathcal{\ell}^x_{B_R }(z) -  \1_{ 
   \{ \mathcal{\ell}_{ B_R}^x(z) > 0  \} }  \big )     \bigg]  \leq 
   \begin{cases} 
\displaystyle  1 -     \frac{c}{\log (R - |x|)} & \mbox{ if $d=3$} \\
\displaystyle 1 - c & \mbox{ if $d>3$}.
   \end{cases}
$$
\end{Proposition} 
\begin{proof} 
We use the fact that 
$\mathbb{E} \big [
  \sum_{z \in \partial B_R  } 
  \mathcal{\ell}^x_{ B_R }(z) \big ]  = 1$ 
 and  apply
Lemma \ref{lemma:connectionsquarelocalsettl} 
with $\delta =  1 /  ( 2  \mathbb{E}[  \Gamma^x_{B_R} )]$
to obtain 
$$
\mathbb{E}
\Big [  
  \sum_{z \in \partial B_R  } \big ( 
  \mathcal{\ell}^x_{ B_R }(z) -  \1_{ 
   \{ \mathcal{\ell}_{ B_R }^x(z) > 0  \} }  \big )     \Big] 
   \leq  1 - \frac{1}{4 \mathbb{E}[ \Gamma^x_{B_R} ]}.
   $$
   The claim then follows from Corollary 
   \ref{cor:localtimeboundaryball}.
\end{proof}

   \subsection{The Random Barrier Growth}\label{sec:RBG}
We now introduce the central object for bounding the growth of the aggregate, namely the \textit{Random Barrier Growth} (RBG).
In RBG, settling occurs only on a spherical layer 
 chosen randomly within a shell of width $H$ according to a certain distribution. 
The choice of this distribution
fulfils a key property, namely
the  expected number of offspring hitting 
 the random
layer within the current aggregate is proportional to the volume of the aggregate over $H^d$. By choosing $H$ appropriately, 
we obtain a contraction on the number of pioneers which do not settle.  The key estimate is presented in
Proposition \ref{prop:settlmentsphericalprocess}
below.

We now proceed with the formal definition of the RBG
in a shell. 
\begin{Definition}[\textbf{Random Barrier Growth}]
\label{def:RBG}
Let $R_1 < R_2$ be two integer numbers,
set  $H = R_2 - R_1$. 
Let $\eta : \Z^d \to  \mathbb{N}$
be the initial particle configuration
with support in $\partial B_{R_1}$
and number of particles $N_1 = \sum_{z \in 
\partial B_{R_1}} \eta(z)$. 
Let further $Z_1, \ldots Z_{N_1}$ be i.i.d.\  random variables with distribution 
 supported in   $ \{1, \ldots,  H\} $  given by
\begin{equation}\label{eq:probabilitydistributionZ}
P( Z_i = h) = \frac{ h^{d-1} }{\sum_{r =1 }^H  r^{d-1}},
\end{equation}
The $\{ Z_i\}$ represent the sequence of heights at which the BRWs will probe the space for settlement.

The RBG consists of a discrete time
Markov chain  $(A_t)_{t \in \mathbb{N}}$
and a sequence of random variables 
$(  \tilde G_t,  \tilde R_t)_{  t \in \mathbb{N}_0  }$,
 defined as follows. 
We set $A_0 =  \emptyset$ 
and label the particles in $\eta$ with integers from $1$ to $N_1$ in some arbitrary order. 
At each step $t \in \{1, \ldots, N_1\}$, 
we start, from the $t$-th particle of $\eta$, an independent BRW
stopped on $\partial B_{ R_1 + Z_t  }$ 
and let $\xi_t$ be the configuration of pioneers  on $\partial B_{ R_1 + Z_t  }$.
We define the aggregate of the step $t$ as 
the union of the aggregate of the step $t-1$ and 
the support of the  pioneers configuration,  namely 
$$
A_t := A_{t-1} \cup \{ z \in \partial  B_{ R_1 + Z_t  } \, \, : \,   \xi_t(z) > 0     \}.
$$
After that we proceed as follows.
\begin{enumerate}
\item[(a)]
For each particle reaching $\partial B_{ R_1 + Z_t  }$ within 
$A_{t-1}$ we start an independent critical branching random walk stopped on
$\partial B_{ R_2  }$  and call $\tilde{G}_t$ the total number of pioneers 
on $\partial B_{ R_2  }$, which we dub \textit{green particles}.
\item[(b)]  
For each vertex 
 $ x \in \partial  B_{ R_1 + Z_t  } \setminus A_{t-1}$
 with $\xi_t(x)>0$,
 we let exactly one of the $\xi_t(x)$ particles at $x$ settle
 and  start 
 $(\xi_t(x) - 1) \1_{ \{ \xi_t(x)>0 \}  }$
  independent critical branching random walks
 stopped on $ B_{ R_2  }$ from $x$, 
 namely one critical Branching Random Walk from each particle which has not settled at $x$. 
 We call $\tilde{R}_t$ 
 be the total number of pioneers on $\partial B_{ R_2  }$,
 which we dub \textit{red particles}.
\end{enumerate}
We then move to the step $t+1$ and iterate the procedure
up to reaching $t = N_1$. 
\end{Definition}
Thus, 
\begin{equation}\label{eq:numberescaping}
N_2 := \sum\limits_{ t =1   }^{N_1}  (\tilde{G}_t + \tilde{R}_t)
\end{equation}
corresponds to the total number of particles reaching 
 $\partial B_{ R_2  }$
 in RBG,  which we refer to as the \textit{escaping particles.}
The next proposition provides an upper bound on the expected number of escaping particles. 
\begin{Proposition}
\label{prop:settlmentsphericalprocess}
Let $d \geq 3$.  
There exist constants $c>0$, $J< \infty$ depending only on the dimension  such that the following holds.
Consider a RBG with integer  radii $R_2 >  R_1$,
starting from an initial particle configuration $\eta$
with support in $ \partial {B}_{ R_1 }$
and with $N_1 = \sum_{x \in  \partial {B}_{ R_1 }} \eta(x) $ particles.  
Then, with $H : = R_2 - R_1$, 
$$
\mathbb{E} \big [  N_2   \big]  \leq  N_1 \,  \Big ( 
J  \,  \frac{N_1}{H^{d}} + \, \,  \gamma_{H} \Big ) 
$$
where
$$
\gamma_{H}  := 
\begin{cases}
    1   - \frac{c}{\log H} & \mbox{ if $d=3$} \\
1 - c  & \mbox{ if $d>3$}.
\end{cases}
$$
\end{Proposition}
This result  is  useful since,  by setting the width of the shell 
$H$ appropriately as a function of $N_1$,  it implies 
that only a fraction of the number of the particles of the initial particle configuration escapes without settling.  
\begin{proof}
We bound the expectation of $\tilde{G}_t$ and $\tilde{R}_t$
in (\ref{eq:numberescaping}) separately, starting from $ \tilde{G}_t$.
Let $t \in \{1, \ldots N_1\}$ be an arbitrary step of the procedure and let  $y_t \in B_{R_1}$ be  the vertex where the $t$-th particle
of the initial configuration is located.
Let $\mathcal{F}_{t}$ be the $\sigma$-algebra
generated by the first $t$ steps of the RBG.

The first observation is that the expectation of $\tilde{G}_t$,
the number of  escaping green particles at the step $t$,
corresponds to the expected number of particles hitting $\partial B_{ R_1+Z_t }$
within $A_{t-1}$.
This holds since,  once particles are stopped in  $\partial B_{ R_1+Z_t } \cap A_{t-1}$, they proceed to $\partial B_{R_2}$ 
as independent BRWs stopped on $B_{R_2}$.
Hence,  
summing over all possible realisations of the random radius at the step $t$
 we deduce that
\begin{equation}\label{eq:randomness}
\begin{split}
\mathbb{E} \big[  \tilde{G}_t  \bigm  | \mathcal{F}_{t-1}  \big ] &  =  
\sum\limits_{h = 1}^H
P( Z_t = h) 
\sum\limits_{ x \in A_{t-1}   \cap  \partial   B_{ R_1 + h } } 
G_{   R_1 + h } (y_t, x  ) \\
& \leq   
\sum\limits_{h = 1}^H
 \frac{ h^{d-1} }{H^d} | A_{t-1}   \cap  \partial  B_{R_1 +  h } | 
\frac{c }{h^{d-1}}  
 \leq   c  \frac{ | A_{t-1}  |}{H^{d}} . 
\end{split}
\end{equation} 
Here we used  (\ref{eq:probabilitydistributionZ}),
Lemma \ref{AG-hitting},
and 
the fact that  the sets $\{ \partial B_{R_1 + h}$  : 
$h \in \mathbb{N} \}$ 
are disjoint. 
Using that $\E [|A_{t-1} | ] \leq t-1$, we deduce that $\E [ \tilde G_t ] \leq ct / H^d$. 

For the  expected number of red particles 
at the step $t$ we have
$$
\mathbb{E} \big[  \tilde{R}_t  \bigm  | \mathcal{F}_{t-1}  \big ]   =  
\frac{1}{  \sum_{r=1}^{H} r^{d-1}}
\sum_{h  = 1 }^H h^{d-1}   \sum\limits_{ z \in   \partial B_{ R_1 + h } \setminus A_{t-1}    } \mathbb{E} \Big  [  ( \xi_t(z) - 1 ) \1_{  \xi_t(z) > 0} \Big ]   
  \leq  \gamma_{H},
$$
where for the last step we summed over all vertices in $ \partial B_{ R_1 + h }$  and applied Proposition~\ref{prop:fractionsettle}.

Combining the two estimates above we obtain 
\begin{align*}
\mathbb{E}[  N_2  ] &  = \sum\limits_{t=1}^{N_1} \mathbb{E}  \big[\tilde{G}_t + \tilde{R}_t \big]  
 \leq  \frac{c}{H^{d}} \sum\limits_{t=1}^{N_1} t \,  + \, \gamma_HN_1 
  \leq  c \frac{N_1^2}{ H^{d}} + \, \, \gamma_H N_1 . 
\end{align*}
This concludes the proof.
\end{proof}

\begin{Remark}\label{rem:layer}
The choice of the barrier distribution in (\ref{eq:probabilitydistributionZ}) 
is crucial for Proposition \ref{prop:settlmentsphericalprocess}. 
Thanks to this distribution, the proportionality depends only on the  size of $A_t$, 
and not on its specific shape. 
\end{Remark}

    \subsection{RBG through shells} \label{sec:RBGshells}
The proof of Theorem \ref{thm:controlling}
relies on analyzing a sequence of shells
and a sequence of RBG's crossing each of them.
In this procedure, the width of each subsequent shell
is not fixed but random: it depends on the number of particles
that escape from the previous shell.
In light of Proposition \ref{prop:settlmentsphericalprocess},
the width is chosen so that, in expectation,
only a fraction of the particles escaping the previous shell
will escape the next one.
This results in a contraction at a rate that can be uniformly bounded up to a certain blow-up time. 
As a consequence,  the total number of shells crossed before the process dies out is small.

\begin{Definition}[\textit{Random Barrier Growth through shells}]
\label{def:RBGshell}
Fix an initial (integer) radius $R_0$ and
an initial particle configuration $\eta_0$ with support on 
 $\partial B_{ R_0}$.
Let $N_0 := \sum_{x \in \mathbb{Z}^d} \eta_0(x)$ be the total number of particles at time zero. 
At each step $t \in \mathbb{Z}_+$,
given the step $t-1$, 
we set
$
R_t := R_{t-1} +   H_t,$
where 
\begin{equation}\label{eq:choiceHt}
H_{t}   : = \Big   \lceil  \,     \Big(  \kappa \,  N_{t-1}  \,   \big( \log (N_{t-1} + 1) \big)^{\beta}  \Big )^{ \frac{1}{d}}  \, \Big \rceil
\end{equation}
with $\beta= 1$ if $d=3$ and $\beta=0$ if $d > 3$. Here $\kappa>0$  is a large enough constant depending only on the dimension.
We start a RBG with radii $R_{t-1}$ and $R_t$ from the  particle configuration 
$\eta_{t-1}$ 
(which has support in $\partial B_{R_{t-1}})$
and let $\eta_t$ be the escaping particle configuration of such process 
(which has  support
in $\partial B_{R_t}$,  recall Definition \ref{def:RBG}).
We then move to the next step $t+1$ and repeat the procedure. 
Define for any $\alpha >0$ the  stopping time (with the convention that $\inf \emptyset = + \infty$)
\begin{align}
\label{eq:Tblow}
T^{(\alpha)} & : = \inf \{ t  \geq 0  \, : \, N_t  > \alpha  N_0  \} . 
\end{align} 
\end{Definition}
The next result shows that, as long as $\{t < T^{(\alpha)}\}$,
the number of particles in the RBG decays exponentially fast in the number of shells which have been crossed, with a rate which might depend on $N_0$, on $\alpha$,  and on the dimension $d\geq 3$. 
\begin{Lemma}\label{lemma:decreaseexpectation}
Let $d \geq 3$. 
In (\ref{eq:choiceHt}) we can choose $\kappa$ large enough, depending only on the dimension, so that for any $\alpha > 1$ and 
 $t \in \mathbb{Z}_+$ we have 
 \begin{equation}
 \label{eq:secondrecursion}
\mathbb{E} \big [  N_t \,   \1_{t \leq T^{(\alpha)}  }    \big ]   \leq 
\begin{cases}
   e^{  - \frac{c \, t \, }{\log ( \alpha N_0 )}}  N_0 & \mbox{ if $d=3$} \\
e^{-c \,  t  \, }  N_0 & \mbox{ if $d>3$.}
\end{cases}
\end{equation}
\end{Lemma}
\begin{proof}
Let $\mathcal{G}_t$ be the sigma-field
generated by the fist $t$ steps of the RBG through Random Shells
and observe that $N_t$,
$H_{t+1}$,  
$R_{t+1}$ and 
$\1_{\{T^{(\alpha)} \geq t+1   \}}$
are $\mathcal{G}_t$-measurable by construction.
B Proposition \ref{prop:settlmentsphericalprocess} 
\begin{align*}
  \mathbb{E} \big  [  N_{t} \1_{\{ T^{(\alpha) } \geq t \}} \big ]  & 
  = 
  \mathbb{E} \Bigg[   \mathbb{E} \big [
N_t \bigm | \mathcal{G}_{t-1} \big ] \1_{\{ T^{(\alpha) } \geq t\}} \Big ] \\
& \leq  
\mathbb{E}
 \bigg [ 
 \Big (  J  \frac{N_{t-1}^2}{H_t^d}
 \, +  \, 
\gamma_{H_t}  N_{t-1}   \Big ) 
 \1_{\{ T^{(\alpha) } \geq t \}} \bigg ] \\
 & \leq  
\mathbb{E}
 \bigg [ N_{t-1} 
 \Big (  \,   \frac{J}{ \kappa} \cdot \frac{1}{ (\log( N_{t-1} + 1))^{\beta}}  \,  +
\gamma_{H_t}     \Big ) 
 \1_{\{T^{(\alpha)}\geq t\}} \bigg ]  \\
  & \leq  
\begin{cases}
\displaystyle 
\mathbb{E} 
 \Big[ N_{t-1} e^{ -  \frac{c}{ \log (N_{t-1} + 1)  }   }
  \1_{\{T^{(\alpha)}\geq t\}} \Big]  & \mbox{ if $d=3$} 
  \\ \displaystyle 
e^{-c}  \, \,   \mathbb{E}
 \big [ N_{t-1}  
 \1_{\{T^{(\alpha)}\geq t\}} \big ]  & \mbox{ if $d>3$} 
\end{cases}
 \\
  & \leq  
\begin{cases}
  \displaystyle 
  e^{ -  \frac{c}{ \log (\alpha N_{0} ) }}
  \mathbb{E} 
 \big [ N_{t-1}  
  \1_{\{ T^{(\alpha)} \geq t-1\}} \big]  & \mbox{ if $d=3$} \\
 e^{-c}  \, 
  \mathbb{E}
 \big [ N_{t-1}  
 \1_{\{ T^{(\alpha)} \geq t-1\}} \big ]   & \mbox{ if $d>3.$} 
\end{cases}
\end{align*} 
For the third inequality above
in dimension $d=3$ we used that $\kappa$ is a large enough constant such that if $N_{t-1}> 0$, recalling that 
	\[ \gamma_{H_t} = 1   - \frac{c}{\log H_t } 
	\quad \mbox{ with } 
	\quad H_{t}   : = \Big   \lceil  \,     \Big(  \kappa \,  N_{t-1}  \,  \log (N_{t-1} + 1)  \Big )^{ \frac{1}{d}}  \, \Big \rceil , 
	\]
 we have
\begin{multline}
\gamma_{H_t}  +  \frac{J}{ \kappa} \cdot \frac{1}{ \log( N_{t-1} + 1)}  \, 
  \leq 
1 - \frac{c d}{\log \kappa + \log (N_{t-1}+1) + \log \log (N_{t-1} + 1)} +  \frac{J}{ \kappa} \cdot \frac{1}{ \log( N_{t-1} + 1)}  \\  \leq 1 - \Big  [  \frac{c \, d }{  2 \log \kappa + 2 } -  \frac{J}{ \kappa}  \Big ] 
\cdot \frac{1}{ \log( N_{t-1} + 1)} \leq e^{  - \frac{c}{ \log ( N_{t-1} +1)  }  }.
\end{multline}
The proof then follows by iterating this bound. 
\end{proof}
The previous lemma
is useful to stochastically bound from above the decay in the number of particles in the RBG when the number of crossed shells is large.  
However, when this number is small, the previous lemma is powerless, and we need the following proposition, which relies on a stochastic comparison with independent BRWs restricted to a ball and  large deviation estimates.  

\begin{Proposition}\label{prop:largedeviationcontrol}
There exists $c >0$ such that, 
for any large enough  $\alpha$,
uniformly in 	$t, D, R_0 \in \mathbb{Z}_+$ satisfying $R_0<D$,  in the initial 
particle configuration
$\eta_0$ with $N_0$ particles and support in $\partial B_{R_0}$,
the RBG 
starting from the particle configuration $\eta_0$ 
satisfies
$$
\mathbb{P}(  N_t > \,  \alpha  \, N_0,  R_{t} < D  ) \leq D \, 
\exp\Big( -  c \,\alpha \,  \frac{N_0}{D^2}\Big).
$$
\end{Proposition}
\begin{proof}
Fix  arbitrary positive integers $D$ and $t $.
By a union bound and the fact that $R_t \in \N$
\begin{equation}\label{eq:stepinitiallargedeviation}
\mathbb{P}(  N_t > \,  \alpha  \, N_0,  R_{t} < D  ) 
\leq 
\sum\limits_{r=R_0}^{D-1} \mathbb{P}(  N_t > \,  \alpha  \, N_0,  R_{t} = r   ).
\end{equation}
We now use that,
almost surely, we can dominate the RBG process
by independent critical BRWs.
Order the particles of the initial configuration $\eta_0$ from $1$ to $ N_0$ arbitrarily.
Fix $r \in \mathbb{Z}_+$. 
From each particle 
we start a independent BRW restricted 
 to $B_{R_0 + r}$,
and  let $\ell^{(i)}_{B_{R_0 + r}}(\partial B_{R_0+r})$
denote the number of pioneers reaching the outer boundary of the ball  $B_{R_0 + r}$.
Moreover, let   
	\[ \mathcal{N}_r := \sum_{i=1}^{ N_0  }
\ell^{(i)}_{B_{R_0 + r}}(\partial B_{R_0+r}) \]
denote the total number of pioneers reaching $\partial B_{R_0+r}$.
The $N_0$ independent BRWs 
starting from each particle in $\eta_0$ 
and the RBG through random shells starting from $\eta_0$
are defined on the same probability space and can be sampled by using the same instructions.
Hence,  by monotonicity (see  Corollary \ref{cor-free} in Appendix \ref{app:Abelian}), we have that
$$
\{ N_t > \alpha N_0, R_t = r  \}
\subseteq 
\{  \mathcal{N}_r > \alpha N_0, R_t = r  \}
\subseteq  
\{  \mathcal{N}_r > \alpha N_0  \} . 
$$
By applying 
 Proposition  \ref{prop:largedeviation}
 with small enough positive  $\lambda$, 
depending only on the dimension,  
we obtain from   from (\ref{eq:stepinitiallargedeviation}) that
\begin{align*}
\mathbb{P}(  N_t > \alpha  \, N_0,  R_{t} < D    )
& \leq 
\sum\limits_{r=R_0}^{D-1} \mathbb{P} \big ( \mathcal{N}_r > \alpha N_0  \big )  \\
& \leq 
\sum\limits_{r=R_0}^{D-1}
e^{  \frac{  - \lambda \alpha N_0  }{   (R_0 + r)^2 } }
\prod_{i=1}^{N_0}
\mathbb{E} \Big [   e^{  \frac{\lambda}{(R_0 + r)^2  } 
\ell^{(i)}_{B_{R_0 + r}}(\partial B_{R_0+r})} \Big ] 
\\
& \leq 
\sum\limits_{r=R_0}^{D-1}
e^{   \frac{ \lambda  (   - \alpha + c)  N_0 }{ (R_0 + r  )^2}} \leq  D
e^{   \frac{  - \alpha  \lambda N_0 }{ 2 D^2}} 
\end{align*}
where for the fourth
inequality we  used that $R_0 + r \leq  D$
and that $\alpha$ is large enough. 
\end{proof}

\subsection{Proof  of Theorem \ref{thm:controlling}} \label{sec:controlling}
We can now conclude the proof of  Theorem \ref{thm:controlling}. 
\begin{proof}[Proof of Theorem \ref{thm:controlling}]
Let $\alpha >1$ be a parameter to be chosen later large enough,  and recall from \eqref{eq:Tblow} the definition of $T^{(\alpha )}$. For some fixed constant $\kappa >0$ chosen large enough so that Lemma \ref{lemma:decreaseexpectation} applies,  define   
$$
t^* :=  
\big\lceil  \alpha^{\frac{1}{d}}  (\log N_0)^{1 + \beta} \big \rceil 
\quad  \quad H^* := \big \lceil   {  {   \big( 2  \alpha  \kappa \,  N_0  \,   (\log  N_0 )}^{ \beta} \big)}^\frac{1}{d} \big \rceil \quad \quad D^* := R_0 + t^* H^* . 
$$
Here $\beta =1$ if $d=3$ and $\beta=0$ if $d>3$, as  in (\ref{eq:choiceHt}).
These choices are such that, for any $\alpha >1$, and for any sufficiently large $N_0$                                                                                                      
\begin{equation}\label{eq:implication}
t <  T^{(\alpha)} \implies R_t \leq R_0 + t  \, H^*.
\end{equation}
For any particle configuration $\eta_0$ with support in $\partial B_{R_0}$ 
and sufficiently many particles, if 
	\[ T_{end}   : = \inf \{ t  \geq 0  \, : \, |\eta_t| = 0  \}, \]
we then have that
\begin{equation}
\label{eq:initialstep}
\begin{split}
\mathbb{P} ( \mathscr{S}(\eta_0)  \not\subseteq  B_{ D^*} )
& \leq  
\mathbb{P}  (R_{T_{end}} > D^*)  \\
& \leq 
\mathbb{P}  (N_{t^*}>0, T^{(\alpha)} \geq  t^*) + \mathbb{P}  (T^{(\alpha)} < t^*) 
\end{split}
\end{equation}
where for the first inequality we used Corollary \ref{cor-free}
to compare Branching IDLA with RBG,
and for the second  inequality we used  (\ref{eq:implication}).
By a direct application of Lemma \ref{lemma:decreaseexpectation} we obtain that for any $\alpha    >1$ and any sufficiently large $N_0$,  
\begin{equation}\label{eq:step1ub}              
\begin{split}
\mathbb{P}  (N_{t^*}>0, T^{(\alpha)} \geq t^*)  & \leq 
 \mathbb{E}[  N_{t^*} \1_{ \{  T^{(\alpha)} \geq  t^* \}  }   ] 
 \leq \begin{cases}
 N_0 \, e^{ - c \frac{t^*}{\log (\alpha N_0)}   }  & \mbox{ if $d=3$ } \\
 N_0  \,  e^{ - c {t^*}   }  &  \mbox{ if $d>3$ }
\end{cases} \\
& \leq
N_0^{ - c \,   \alpha^{\frac{1}{d}}  } . 
\end{split}
\end{equation}
By definition of ${T^{(\alpha)}}$ and by (\ref{eq:implication}), this shows that 
we can choose $\alpha$ large enough so that,  for any sufficiently large $N_0$, 
\begin{equation}\label{eq:step2}
\begin{split}
\mathbb{P}  (T^{(\alpha)} < t^*)  
& \leq 
\mathbb{P} ( \exists t \in \{ 1, \ldots, {t^*-1}  \} \, : \,  R_t \leq  D^*, N_t > \alpha N_0   ) \\
& \leq 
\sum_{t=1}^{   t^*    } 
\sum_{r=R_0}^{D^*} 
\mathbb{P} (   R_t =r,  N_t > \alpha  N_0   ) 
 \leq t^*  (D^*)^2 e^{  -  c \frac{\alpha N_0 }{ (D^*)^2  }  } \\
 &  \leq  N_0 \,  (\log(N_0+1))^{8} 
 e^{  -  c  \, \, \alpha^{ \frac{d-2}{d} } \log N_0 }     
 \end{split}
\end{equation}
where for the third inequality we used Proposition \ref{prop:largedeviationcontrol},
for the last inequality we used the assumption that
$(R_0 )^2 \leq  \frac{N_0}{  \alpha^{\frac{2}{d} }\log N_0 }$,
which implies $(D^*)^2 \leq   \frac{ 2 N_0}{ \,  \alpha^{\frac{2}{d}} \,  \log N_0}$ for any $\alpha >1$ and any sufficiently large $N_0$.
Plugging  (\ref{eq:step1ub}) and (\ref{eq:step2}) into (\ref{eq:initialstep}) we then obtain that,
for any $\alpha$  and $N_0$ sufficiently large, 
 $$
\mathbb{P}\Big ( \mathscr{S}(\eta_0) \not\subseteq B_{ R_0 +  c  \,  \alpha^{  \frac{2}{d} } {N_0}^{\frac{1}{d}} (\log N_0)^{ 1 +  \frac{7}{2} \beta }   }      \Big )  \leq 
\frac{c}{{N_0}^{ \alpha^{\frac{1}{d} }}} , 
$$
which concludes  the proof. 
 \end{proof}
 
 \subsection{Proof of Theorem \ref{th:deviation}}
 \label{sec:finalsection}
In this final section we prove Theorem \ref{th:deviation}, of which the shape theorem in Theorem \ref{th:main} is a direct corollary. 
 \begin{proof}
 Suppose that $d \geq 3$, 
 and let  $\epsilon \in (0, \frac{1}{2}), \delta >1$ be arbitrary. 
We let $\eta = |B_n| \1_0$ denote  the initial particle configuration with $|B_n|$ particles at the origin, for  $n \in \mathbb{Z}_+$. 
By Proposition \ref{prop:lowerbound},   for any sufficiently large $n$
   \begin{equation}\label{eq:firstestimate}
\mathbb{P}  \big ( B_{  n - n^{1/2+ \epsilon} } \not\subseteq    \mathscr{S}(\eta )  \big )  \leq n^d \exp \Big(  - c  \frac{n^{2 \epsilon}}{\log n} \Big) \leq \frac{1}{n^{\delta}} . 
   \end{equation}
   Moreover,  by conditioning on the number of particles reaching the outer boundary of $B_n$ when starting BIDLA from $\eta $, we   
    have  that 
   \begin{multline}
   \label{eq:secondestimate}
         \mathbb{P}  \Big (   \mathscr{S}(\eta ) \not\subset B_{ n + n^{1 - \frac{1}{2d} + \epsilon}  } \Big ) \\ \leq 
         \mathbb{P}  \Big (  \mathscr{S}(\eta ) \ \not\subset B_{n + n^{1 -\frac{1}{2d} + \epsilon}  } \,  \Big | \,  \cS^{\eta }_{B_n}(\partial B_n) \leq  n^{d-\frac{1}{2}+ \frac{ d \epsilon}{2}}  \Big ) +     \mathbb{P}  \Big ( \cS^{\eta }_{B_n}(\partial B_n)>  n^{d-\frac{1}{2}+ \frac{ d \epsilon}{2}} \Big ) \\
          \leq  \frac{1}{n^{\delta}} 
          + n^d \exp \Big(  - c \cdot  \frac{n^{2  d \epsilon}}{\log n} \Big) 
          \leq \frac{2}{n^{\delta}}
     \end{multline}
for any  $n$ large enough. 
In the last inequality  we used the  Abelian property and applied
Theorem \ref{thm:controlling} with $R=n$ and 
$N= |  \cS^{\eta }_{B_n}(\partial B_n)  | \leq   n^{d- \frac{1}{2} + \frac{d \epsilon}{2}}$
in order to upper bound the first term, 
while for the second term we used Proposition \ref{prop-standby}.
The bounds of Theorem \ref{th:main} when $d \geq 3$ and of Theorem \ref{th:deviation} then follow from  (\ref{eq:firstestimate}), (\ref{eq:secondestimate}) and an application of the  Borel-Cantelli Lemma.
 \end{proof} 
 
 \appendix 
 
\section{The Abelian property} \label{app:Abelian}
One of the key properties of a large family of related models is the so-called Abelian property, which allows particles to be moved in any order -- provided certain rules are respected -- without altering the statistics of the final configuration.
In this section we establish this property for Branching IDLA, following the framework of Rolla and Sidoravicius \cite{rolla2012absorbing},  and adapting their definitions and arguments to our context.

\paragraph{Particle configurations.}
Given any particle configuration $\eta: \Z^d \to \mathbb N$,
we say that  the vertex $x\in \mathbb Z^d$ is \textit{stable} in $\eta$ if $\eta(x)\le 1$, and it is  \textit{unstable} otherwise. 
The particle configuration $\eta$ is stable if every vertex $x \in \mathbb{Z}^d$ is stable in $\eta$.
We interpret $\eta(x)=0$ as the site $x$ being empty, and $\eta(x)=1$ as the particle at $x$ being settled.

\paragraph{Instructions.} 
To state the Abelian property for Branching IDLA, we first introduce a collection of instructions, that is, operators acting on a configuration of particles and producing a new configuration.
When a particle performs a Branching Random Walk step, it dies and creates offspring at its neighbouring sites.  
Accordingly, we introduce birth-death instructions.
In Branching IDLA, a particle performs a BRW step only if it has not yet settled, that is, if the site currently hosts at least two particles.
In this case, we refer to the instruction as \emph{legal}.
For the purpose of deriving a Least Action Principle and comparing Branching IDLA to the RBG process, however, it will be convenient to also allow settled particles to move.  For that reason, we allow to use instructions at sites containing at least one particle,  in which case the instruction is referred to as \emph{acceptable}.

A \textit{birth-death instruction},  $\mathfrak A^x_\xi$, is defined through a vector $\xi \in \bigcup_{n\ge 0} \mathcal{N}^n$,
where $\mathcal{N}$ is the set of sites which are neighbours of the origin. 
If $\xi$ has $k$-coordinates,
each entry is thought of as the step that one of the $k$ children of an individual at $x$ makes. More precisely,  for any 
$\xi=(\xi(1),\dots,\xi(k)) \in \bigcup_{n\ge 0} \mathcal{N}^n$  with $k $ coordinates we  have  
\[
\text{if}\quad \eta(x)\ge 1,\quad\mathfrak A^x_\xi\text{ is $\eta$-acceptable and}\quad 
\mathfrak A^x_\xi(\eta) : =  \eta-\1_x +\1_{k>0}\cdot \sum_{i=1}^k \1_{x+\xi(i)}.
\]
In words, one individual at $x$ produces $k\ge 0$ children and dies, and the child numbered $i\le k$
is born at $x+\xi(i)$. 
Let us denote by $\mathcal{I}_x$ the set of all possible birth-death instructions at $x$,
namely $$\mathcal{I}_x = \Big\{  \mathfrak A^x_\xi  \, : \, \xi \in \bigcup_{n \geq 0} \,   \mathcal{N}^n   \Big\}.$$
A birth-death instruction at $x$,   $\mathfrak A^x \in \mathcal{I}_x$,
 is $\eta$-\textit{legal} 
if $\eta(x)>1$, $\eta$-\textit{acceptable} if $\eta (x) \ge 1$  and it is $\eta$-forbidden if $\eta(x)=0$.

The two key properties which make this representation Abelian are that, for any 
$\eta$-acceptable  birth-death instruction at $x$,   $\mathfrak A^x \in \mathcal{I}_x$,  and any  $n \in \mathbb N$, 
\begin{equation}\label{harris-1}
\mathfrak A^x\Big(\eta+n\1_x\Big)=\mathfrak A^x(\eta)+n\1_x,
\end{equation}
and that, for any $y\not= x$, and an $\eta$-acceptable birth-death instruction at $y$, $\mathfrak A^y \in \mathcal{I}_y$,  we have
\begin{equation}\label{harris-2}
\mathfrak A^y(\eta)(x)\ge \eta(x).
\end{equation}
These two properties have the following implication.
Assume that $\eta(x)\ge 1$ (and, by \eqref{harris-2}, this implies that $\mathfrak A^x$ is $\mathfrak A^y(\eta)$-acceptable), 
and that $\mathfrak A^y$ is $\eta$-acceptable. Then $\mathfrak A^y$ is $\mathfrak A^x(\eta)$-acceptable by \eqref{harris-2}, 
and
\begin{equation}\label{harris-key}
\mathfrak A^x\circ \mathfrak A^y(\eta)=\mathfrak A^y\circ \mathfrak A^x(\eta).
\end{equation}

\paragraph{Stack of instructions and sequence of topplings.}
In this framework, the instructions are assigned to sites rather than to particles. 
For this reason, we introduce \textit{stacks of instructions}, that is, sequences of birth--death instructions associated with each site. 
More formally, a stack of instructions is a vector 
\[
\tau = (\tau_{x,j})_{x \in \mathbb{Z}^d,\, j \in \mathbb{Z}_+},
\]
where each element $\tau_{x,j}$ belongs to $\mathcal{I}_x$, the set of instructions associated with the site~$x$.
In order to sample a Branching IDLA process using stacks of instructions,  
the elements $ (\tau_{x,j})_{x \in \mathbb{Z}^d,\, j \in \mathbb{Z}_+}$
are drawn at random as i.i.d.\  random variables 
  such that, for each birth-death instruction
 $ \mathfrak A^x_\xi \in \mathcal{I}_x$ producing $k \in \mathbb{N}$ children, i.e.\  $\xi \in \mathcal{N}^k$, 
	\[ 
	\P ( \tau_{x,j} =  \mathfrak A^x_\xi  ) = \frac{1}{(2d)^k} \nu(k),
	\]
where we recall that  $\nu $ denotes the offspring distribution.
Given a stack of instructions $\tau$, we now define {sequences of topplings}. 
Starting from an arbitrary initial particle configuration $\eta : \Z^d \to  \mathbb{N}$, we apply the instructions according to an ordered sequence of sites 
$
\alpha = (x_1, x_2, \ldots, x_n) .
$
To keep track of the number of instructions used at each site, we introduce an \textit{odometer field}
$
h : \mathbb{Z}^d \to  \mathbb{N},
$
which records, for every site, how many instructions have been used. 
A \textit{toppling} at a site $x$
is acceptable if $\eta(x) \geq 1$ 
and legal if $\eta(x) > 1$, and it is defined as 
$$
\Phi_x(\eta,h) = \bigl( \tau_{ x, h(x)+1  }  \eta ,\, h + \mathbf{1}_x \bigr).
$$
A \textit{sequence of topplings} along 
$\alpha$ is acceptable (resp. legal)  if,  for each $ 0 \leq k \leq n-1$,
the particle configuration of 
$\Phi_{x_k} \circ \dots \circ \Phi_{x_1}(\eta,h)$
has at least one particle (resp. two particles) at $x_{k+1}$.
If $\alpha$ is acceptable,  then the sequence of topplings  $\alpha$ is defined by the following operation
\[
\Phi_\alpha(\eta,h) = \Phi_{x_n} \circ \dots \circ \Phi_{x_1}(\eta,h).
\]
We omit writing the odometer field $h$ when it is not needed, and think of $\Phi_\alpha $ as acting on configurations.
For any  sequence $\alpha$,   the  {\it odometer} of $\alpha$,
$m_\alpha$,  is defined as $$m_\alpha=\sum_{x\in \alpha} \1_{x}.$$ 
Thus $m_\alpha$ counts the number of appearances of the site $x$ in the sequence $\alpha$. 
\begin{Lemma}\label{lem-Abelian-1}
If $\alpha,\beta$ are $\eta$-acceptable finite sequences and $m_\alpha=m_\beta$,  then for any odometer field $h$,
$\Phi_\alpha(\eta,h)=\Phi_\beta(\eta,h)$.
\end{Lemma}
\begin{proof}  We write a proof by induction on the length of $\alpha$, that we denote with $|\alpha|$.
The property is obviously true for $|\alpha|=1$, since then $\alpha=\beta$. 
Assume that $|\alpha|=n > 1$ and that 
$\alpha=(x,x_2,\dots,x_n)$, whereas $x$ occurs in position $j+1>1$ for the first
time in $\beta=(y_1,\dots,y_{j},x,\dots)$, and $y_1,\dots,y_j\not= x$. 
Call now $\beta'$ the sequence obtained from $\beta$ by moving  the first occurrence of $x$ to the front, namely $\beta^\prime = (x, y_1, y_2, \ldots, y_{j}, \ldots)$.
We will show that $\beta^\prime$ is acceptable for $\eta$ and that $\Phi_\beta(\eta,h) = \Phi_{\beta^\prime}(\eta,h)$.
Since  the first site of 
$\alpha$ and that of $\beta'$  are the same, and since $m_{\alpha} = m_{\beta^\prime}$, 
we then deduce by induction that $\Phi_\alpha(\eta,h) = \Phi_{\beta^\prime}(\eta,h) = \Phi_\beta(\eta,h) = \Phi_{\beta^\prime}(\eta,h)$.

Since $\Phi_\alpha$ is $\eta$-acceptable we have that $\eta(x)\ge 1$.
Call $\beta_j:=(y_1,\dots,y_{j-1})$, and note that
$\eta(x)\ge 1$ implies that $\Phi_{\beta_j}(\eta)(x)\ge 1$, since,  by property \eqref{harris-2} mass cannot decrease at $x$ when
toppling sites other than $x$.
By applying \reff{harris-key} we deduce that,  if $\mathfrak A^{y_j}$ and $\mathfrak A^{x}$ are the instructions occurring
 in $\beta$ in the $j$ and $j+1$ positions respectively,  then
\begin{equation}\label{Abelian-2}
\mathfrak A^{y_j} \ \text{ is }\mathfrak A^{x}(\Phi_{\beta_j}(\eta))\text{-acceptable}
\quad\text{and}\quad
\mathfrak A^{x}\circ\mathfrak A^{y_j} (\Phi_{\beta_j}(\eta))=\mathfrak A^{y_j} \circ\mathfrak A^{x}(\Phi_{\beta_j}(\eta)).
\end{equation}
Now, by induction, we have that $\Phi_{\beta_{j+1}}$ is $\Phi_{x}(\eta)$-acceptable and
\begin{equation}\label{Abelian-3}
\Phi_{x}\circ \Phi_{\beta_{j+1}}(\eta)=\Phi_{\beta_{j+1}}\circ\Phi_{x}(\eta).
\end{equation}
This implies that $ \Phi_{\beta^\prime}(\eta,h) = \Phi_\beta(\eta,h)$ as desired, thus  concluding the proof. 
\end{proof} 

The second lemma we present is the key result toward the Least Action Principle.
We write 
$m_\beta\prec m_\alpha$
if $m_\beta(x) \leq m_{\alpha}(x)$ for all 
$x \in \mathbb{Z}^d$. 
\begin{Lemma}\label{lem-Abelian-2}
If $\alpha$ is an $\eta$-acceptable sequence of topplings,  $\Phi_\alpha(\eta)$ is stable in the finite set $K \subset \mathbb Z^d$, and $\beta$ is a finite $\eta$-legal sequence of topplings in $K$, then
$m_\beta\prec m_\alpha$.
\end{Lemma}
\begin{proof} 
The proof is obtained by induction on $|\beta|$. The claim is obvious if $|\beta|=1$, so assume $|\eta | >1$.  Let $y\in K$ be
such that $\beta:=(y,y_2,\dots)$.
By assumption, it is necessarily the case that  $\eta(y)>1$. 
Since $\alpha$ stabilizes $\eta$, there must be a toppling which acts on $y$: 
suppose that $y$ occurs for the first time  in position $j+1$ in $\alpha$.
As in the previous proof, in $\Phi_\alpha$ we can commute the instruction $\mathfrak A^y$ with
the first $j$ instructions.  Call $\alpha_1$ the sequence of topplings obtained by deleting the first site $y$ occurring in the sequence $\alpha$, $\beta_1=(y_2,\dots)$ and $(\eta_1,h_1)=\Phi_y(\eta,h)$. 
Now, $|\beta_1|=|\beta|-1$, the configuration $\Phi_{\alpha_1}(\eta_1,h_1)$ is stable in $K$, and $\beta_1$ is a $\eta_1$-legal sequence.
Thus by induction we deduce that $m_{\beta_1}\prec m_{\alpha_1}$,  and since 
$m_{\beta}=m_{\beta_1}+\1_{y} $ and $ m_{\alpha}=m_{\alpha_1}+\1_{y}$, 
we conclude that $ m_\beta\prec m_\alpha $. 
\end{proof} 

We can now state the Least Action Principle, which is a direct corollary of Lemma~\ref{lem-Abelian-2}.
We say that the   sequence $\alpha$, which is acceptable for $\eta$, stabilizes $\eta$ in $K$
if the particle configuration $\Phi_\alpha(\eta)$ is stable in $K$.  
\begin{Proposition}\label{prop-Abelian} If $\alpha$ and $\beta$ are both $\eta$-legal finite sequences which stabilize $\eta$ in $K$, then $m_\alpha=m_\beta$, and consequently $\Phi_\alpha(\eta)=\Phi_\beta(\eta)$.
Furthermore, for any finite $K\subset \mathbb Z^d$, any $K$-stabilizing $\eta$-acceptable finite sequence $\alpha$ has an odometer $m_\alpha$
above $m_\beta$,  where $\beta$ is any $\eta$-legal sequence in $K$ which stabilizes $\eta$ in $K$.
\end{Proposition}
In summary,  any successful $\eta$-legal finite sequence of topplings  produces the same stable configuration, and its odometer is below the odometer of any stabilizing $\eta$-acceptable sequence of topplings.

Lemma~\ref{lem-Abelian-2} leads  to  the following useful corollary,
which is applied in Section~\ref{sect:outerboundgrowth} to compare
Branching IDLA with the Random Barrier Growth.
Indeed,   Branching IDLA  can be sampled by toppling any stabilizing
\emph{legal} sequence, whereas  RBG can be sampled by
toppling an \emph{acceptable} stabilizing sequence.
If the odometer of the RBG is known to vanish outside a large ball
(that is, no particle leaves that ball in the RBG),
which is precisely what we prove, then it follows that the odometer of
Branching IDLA also vanishes outside the same ball
(that is, no particle leaves the same ball in Branching IDLA).
\begin{Corollary}\label{cor-free}
Let $K$ be a finite subset of $\mathbb{Z}^d$,
and let $\eta$ be an arbitrary particle configuration with support contained in~$K$.
Suppose that $\alpha$ is $\eta$-legal  and stabilizes $\eta$ in~$K$,
and that $\beta$ is $\eta$-acceptable and stabilizes $\eta$ in~$K$,
with $m_\beta(x) = 0$ for every $x \in K^c$.
Then $m_\alpha(x) = 0$ for every $x \in K^c$.
\end{Corollary}

 
 \medskip

\bibliography{BIDLAbiblio}

\begin{thebibliography}{10}

\bibitem{asselah2013logarithmic}
Amine Asselah and Alexandre Gaudilli{\`e}re.
\newblock From logarithmic to subdiffusive polynomial fluctuations for
  {I}nternal {DLA} and related growth models.
\newblock {\em Ann. Probab.}, 41(3A):1115--1159, 2013.

\bibitem{asselah2013sublogarithmic}
Amine Asselah and Alexandre Gaudilli{\`e}re.
\newblock Sublogarithmic fluctuations for {I}nternal {DLA}.
\newblock {\em Ann. Probab.}, 41(3A):1160--1179, 2013.

\bibitem{asselah2022time}
Amine Asselah and Bruno Schapira.
\newblock Time spent in a ball by a critical branching random walk.
\newblock {\em Journal de l'Ecole Polytechnique}, 11:1441--1481, 2024.

\bibitem{bak1988self}
Per Bak, Chao Tang, and Kurt Wiesenfeld.
\newblock Self-organized criticality.
\newblock {\em Physical Review A}, 38(1):364, 1988.

\bibitem{benjamini2017internal}
Itai Benjamini, Hugo Duminil-Copin, Gady Kozma, and Cyrille Lucas.
\newblock Internal diffusion-limited aggregation with uniform starting points.
\newblock {\em Electronic Communications in Probability}, 25:Paper No.~46, 12,
  2020.

\bibitem{bond_levine_abelian_networks_I}
Benjamin Bond and Lionel Levine.
\newblock Abelian networks i: Foundations and examples.
\newblock {\em SIAM Journal on Discrete Mathematics}, 30(2):856--874, 2016.

\bibitem{bou2024internal}
Ahmed Bou-Rabee and Ewain Gwynne.
\newblock Internal dla on mated-crt maps.
\newblock {\em The Annals of Probability}, 52(6):2173--2237, 2024.

\bibitem{candellero_stauffer_taggi_oil_and_water}
Elisabetta Candellero, Alexandre Stauffer, and Lorenzo Taggi.
\newblock Abelian oil and water dynamics does not have an absorbing-state phase
  transition.
\newblock {\em Transactions of the American Mathematical Society},
  374(4):2733--2752, 2021.

\bibitem{derrida1992needle}
Bernard Derrida and Vincent Hakim.
\newblock Needle models of laplacian growth.
\newblock {\em Physical Review A}, 45(12):8759, 1992.

\bibitem{diaconis1991growth}
Persi Diaconis and William Fulton.
\newblock A growth model, a game, an algebra, lagrange inversion, and
  characteristic classes.
\newblock {\em Rend. Sem. Mat. Univ. Pol. Torino}, 49(1):95--119, 1991.

\bibitem{duminil2013containing}
Hugo Duminil-Copin, Cyrille Lucas, Ariel Yadin, and Amir Yehudayoff.
\newblock Containing internal diffusion limited aggregation.
\newblock {\em Electronic Communications in Probability}, 18:1--8, 2013.

\bibitem{eden1961two}
Murray Eden.
\newblock A two-dimensional growth process.
\newblock In {\em Proc. 4th {B}erkeley {S}ympos. {M}ath. {S}tatist. and
  {P}rob., {V}ol. {IV}}, pages 223--239. Univ. California Press, Berkeley,
  Calif., 1961.

\bibitem{hastings1998laplacian}
Matthew~B Hastings and Leonid~S Levitov.
\newblock Laplacian growth as one-dimensional turbulence.
\newblock {\em Physica D: Nonlinear Phenomena}, 116(1-2):244--252, 1998.

\bibitem{hoffman_candellero_ganguly_levine_oil_and_water}
Christopher Hoffman, Elisabetta Candellero, Shirshendu Ganguly, and Lionel
  Levine.
\newblock Oil and water: a two-type internal aggregation model.
\newblock {\em Annals of Probability}, 45(6A):4019--4070, 2017.

\bibitem{hoffman2024density}
Christopher Hoffman, Tobias Johnson, and Matthew Junge.
\newblock The density conjecture for activated random walk.
\newblock {\em arXiv preprint arXiv:2406.01731}, 2024.

\bibitem{jerison2012logarithmic}
David Jerison, Lionel Levine, and Scott Sheffield.
\newblock Logarithmic fluctuations for internal {DLA}.
\newblock {\em J. Amer. Math. Soc.}, 25(1):271--301, 2012.

\bibitem{jerison2014internal}
David Jerison, Lionel Levine, and Scott Sheffield.
\newblock Internal {DLA} and the {G}aussian free field.
\newblock {\em Duke Math. J.}, 163(2):267--308, 2014.

\bibitem{jerison2013internal}
David Jerison, Lionel Levine, Scott Sheffield, et~al.
\newblock Internal {DLA} in higher dimensions.
\newblock {\em Electronic Journal of Probability}, 18, 2013.

\bibitem{kesten1987long}
Harry Kesten.
\newblock How long are the arms in dla?
\newblock {\em Journal of Physics A: Mathematical and General}, 20(1):L29,
  1987.

\bibitem{lawler1992internal}
Gregory~F. Lawler, Maury Bramson, and David Griffeath.
\newblock Internal diffusion limited aggregation.
\newblock {\em The Annals of Probability}, 20(4):2117--2140, 1992.

\bibitem{levine2018long}
Lionel Levine and Vittoria Silvestri.
\newblock How long does it take for internal dla to forget its initial profile?
\newblock {\em Probability Theory and Related Fields}, pages 1--53, 2018.

\bibitem{levine2021far}
Lionel Levine and Vittoria Silvestri.
\newblock How far do activated random walkers spread from a single source?
\newblock {\em Journal of Statistical Physics}, 185(3):18, 2021.

\bibitem{levine2024universality}
Lionel Levine and Vittoria Silvestri.
\newblock Universality conjectures for activated random walk.
\newblock {\em Probability Surveys}, 21:1--27, 2024.

\bibitem{losev2025long}
Ilya Losev and Stanislav Smirnov.
\newblock How long are the arms in dbm?
\newblock {\em Communications in Mathematical Physics}, 406(4):1--16, 2025.

\bibitem{lucas2014limiting}
Cyrille Lucas.
\newblock The limiting shape for drifted internal diffusion limited aggregation
  is a true heat ball.
\newblock {\em Probability Theory and Related Fields}, 159(1):197--235, 2014.

\bibitem{meakin1986formation}
Paul Meakin and John~M Deutch.
\newblock The formation of surfaces by diffusion limited annihilation.
\newblock {\em The Journal of Chemical Physics}, 85(4):2320--2325, 1986.

\bibitem{niemeyer1984fractal}
L.~Niemeyer, L.~Pietronero, and H.~J. Wiesmann.
\newblock Fractal dimension of dielectric breakdown.
\newblock {\em Phys. Rev. Lett.}, 52(12):1033--1036, 1984.

\bibitem{norris2023scaling}
James Norris, Vittoria Silvestri, and Amanda Turner.
\newblock Scaling limits for planar aggregation with subcritical fluctuations.
\newblock {\em Probability Theory and Related Fields}, 185(1):185--250, 2023.

\bibitem{norris2024stability}
James Norris, Vittoria Silvestri, and Amanda Turner.
\newblock Stability of regularized hastings--levitov aggregation in the
  subcritical regime.
\newblock {\em Communications in Mathematical Physics}, 405(3):74, 2024.

\bibitem{rolla2020activated}
Leonardo~T. Rolla.
\newblock Activated random walks on $\mathbb{Z}^{d}$.
\newblock {\em Probab. Surveys}, 17:478--544, 2020.

\bibitem{rolla2012absorbing}
Leonardo~T. Rolla and Vladas Sidoravicius.
\newblock Absorbing-state phase transition for driven-dissipative stochastic
  dynamics on $\mathbb{Z}$.
\newblock {\em Invent. Math.}, 188(1):127--150, 2012.

\bibitem{rolla2019universality}
Leonardo~T Rolla, Vladas Sidoravicius, and Olivier Zindy.
\newblock Universality and sharpness in absorbing-state phase transitions.
\newblock In {\em Annales de l'Institut Henri Poincare Physique Theorique},
  2019.

\bibitem{silvestri2020internal}
Vittoria Silvestri.
\newblock Internal {DLA} on cylinder graphs: fluctuations and mixing.
\newblock {\em Electron. Commun. Probab.}, 25:Paper No.~61, 14, 2020.

\bibitem{witten1981diffusion}
TA~Witten~Jr and Leonard~M Sander.
\newblock Diffusion-limited aggregation, a kinetic critical phenomenon.
\newblock {\em Physical Review Letters}, 47(19):1400, 1981.

\end{thebibliography}
\bibliographystyle{plain}

\end{document}